\nonstopmode
\documentclass[10pt]{amsart}
\usepackage{latexsym}
\usepackage{fancyhdr}
\usepackage{amsmath, amssymb,amsthm}
\usepackage{color}
\usepackage[all]{xy}
\usepackage{pdflscape}
\usepackage{longtable}
\usepackage{rotating}
\usepackage{verbatim}
\usepackage{hyperref}
\usepackage{cleveref}
\usepackage{subfigure}
\usepackage{mathrsfs}
\usepackage{mdwlist}
\usepackage{dsfont}
\usepackage{bbold}
\usepackage{mathtools}
\usepackage{booktabs}
\usepackage{float}
\usepackage{tensor}
\usepackage{color}

\newcounter{mnotecount}[section]

\newcommand{\rmnote}[1]{}

\allowdisplaybreaks

\DeclareFontFamily{U}{mathb}{\hyphenchar\font45}
\DeclareFontShape{U}{mathb}{m}{n}{
      <5> <6> <7> <8> <9> <10> gen * mathb
      <10.95> mathb10 <12> <14.4> <17.28> <20.74> <24.88> mathb12
      }{}
\DeclareSymbolFont{mathb}{U}{mathb}{m}{n}
\DeclareFontSubstitution{U}{mathb}{m}{n}

\let\dot\relax
\DeclareMathAccent{\dot}{0}{mathb}{"39}
\let\ddot\relax
\DeclareMathAccent{\ddot}{0}{mathb}{"3A}
\let\dddot\relax
\DeclareMathAccent{\dddot}{0}{mathb}{"3B}
\let\ddddot\relax
\DeclareMathAccent{\ddddot}{0}{mathb}{"3C}

\theoremstyle{plain}
\newtheorem*{theorem*}{Theorem}
\newtheorem{theorem}{Theorem}[section]
\newtheorem*{lemma*}{Lemma}
\newtheorem{lemma}[theorem]{Lemma}
\newtheorem*{proposition*}{Proposition}
\newtheorem{proposition}[theorem]{Proposition}
\newtheorem*{corollary*}{Corollary}
\newtheorem{corollary}[theorem]{Corollary}
\newtheorem*{claim*}{Claim}
\newtheorem{claim}{Claim}

\newtheorem*{conjecture*}{Conjecture}

\newtheorem*{question*}{Question}
\theoremstyle{definition}
\newtheorem*{definition*}{Definition}
\newtheorem{definition}[theorem]{Definition}
\newtheorem*{example*}{Example}

\newtheorem*{algorithm*}{Algorithm}
\newtheorem*{remark*}{Remark}
\newtheorem*{remarks*}{Remarks}
\newtheorem{remark}[theorem]{Remark}

\newtheorem*{convention*}{Convention}

\numberwithin{equation}{section}

\newdir{ >}{{}*!/-10pt/@{>}}
\newdir^{ (}{{}*!/-5pt/@^{(}}

\def\al{\alpha}
\def\be{\beta}
\def\ga{\gamma}
\def\de{\delta}

\def\ze{\zeta}

\def\la{\lambda}

\def\rh{\rho}

\def\si{\sigma}

\def\ta{\tau}

\def\ph{\phi}
\def\vh{\varphi}

\def\ps{\psi}
\def\om{\omega}

\def\Ph{\Phi}

\def\C{\mathbb{C}}

\def\N{\mathbb{N}}

\def\R{\mathbb{R}}

\def\cA{\mathcal{A}}
\def\cB{\mathcal{B}}

\def\cD{\mathcal{D}}

\def\cF{\mathcal{F}}
\def\cG{\mathcal{G}}
\def\cH{\mathcal{H}}

\def\cL{\mathcal{L}}

\def\cN{\mathcal{N}}

\def\cP{\mathcal{P}}

\def\cR{\mathcal{R}}
\def\cS{\mathcal{S}}

\newcommand{\Diff}{\on{Diff}}

\def\rM{\{M\}}
\def\bM{(M)}

\def\BM{\cB^{[M]}}
\def\BrM{\cB^{\rM}}
\def\BbM{\cB^{\bM}}

\def\DM{\cD^{[M]}}
\def\DrM{\cD^{\rM}}
\def\DbM{\cD^{\bM}}

\def\CM{C^{[M]}}
\def\CrM{C^{\rM}}
\def\CbM{C^{\bM}}

\def\SLM{\tensor{\cS}{}_{[L]}^{[M]}}
\def\SrLM{\tensor{\cS}{}_{\{L\}}^{\rM}}
\def\SbLM{\tensor{\cS}{}_{(L)}^{\bM}}

\def\SrML{\tensor{\cS}{}_{\{M\}}^{\{L\}}}

\def\Sp{W^{\infty,p}}

\def\SpM{W^{[M],p}}
\def\SpbM{W^{\bM,p}}
\def\SprM{W^{\rM,p}}

\def\p{\partial}

\renewcommand{\Im}{\mathrm{Im}}
\def\<{\langle}
\def\>{\rangle}
\renewcommand{\o}{\circ}

\def\supp{\on{supp}}
\def\Id{\on{Id}}

\let\on=\operatorname

\newcommand{\sr}[1]%
{\ifmmode{}^\dagger\else${}^\dagger$\fi\ifvmode
\vbox to 0pt{\vss
 \hbox to 0pt{\hskip\hsize\hskip1em
 \vbox{\hsize3cm\raggedright\pretolerance10000
 \noindent #1\hfill}\hss}\vss}\else
 \vadjust{\vbox to0pt{\vss%
 \hbox to 0pt{\hskip\hsize\hskip1em%
 \vbox{\hsize3cm\raggedright\pretolerance10000%
 \noindent #1\hfill}\hss}\vss}}\fi%
}

\def\ev{\on{ev}}

\providecommand{\mapsfrom}{\kern.2em%
\setbox0=\hbox{$\leftarrow$\kern-.10em\rule[0.26mm]{0.1mm}{1.3mm}}\box0%
\kern.3em}

\title{The Trouv\'e group for spaces of test functions}

\author[D.N.~Nenning and A.~Rainer]{David Nicolas Nenning and Armin Rainer}

\address{D.N.~Nenning: Fakult\"at f\"ur Mathematik, Universit\"at Wien, 
Oskar-Morgenstern-Platz~1, A-1090 Wien, Austria}
\email{david.nicolas.nenning@univie.ac.at}

\address{A.~Rainer: 
Fakult\"at f\"ur Mathematik, Universit\"at Wien, 
Oskar-Morgenstern-Platz~1, A-1090 Wien, Austria}
\email{armin.rainer@univie.ac.at}

\begin{document}
	
\begin{abstract}
	The Trouv\'e group $\cG_\cA$ from image analysis consists of the flows at a fixed time of all time-dependent vectors fields 
	of a given regularity $\cA(\R^d,\R^d)$.  
	For a multitude of regularity classes $\cA$, we prove that the Trouv\'e group $\cG_\cA$ coincides with the 
	connected component of the identity of the group of orientation preserving diffeomorphims of $\R^d$ 
	which differ from the identity by a mapping of class $\cA$. We thus conclude that $\cG_\cA$ has a natural 
	regular Lie group structure. 
	In many cases we show that the mapping which takes a time-dependent vector field to its flow is continuous.
	As a consequence we obtain that the scale of Bergman spaces on the polystrip with variable width 
	is stable under solving ordinary differential equations. 
\end{abstract}

\thanks{The authors were supported by FWF-Project P~26735-N25}
\date{\today}

	\maketitle
	
	\section{Introduction}

	It is well-known that a time-dependent vector field $u \in L^1([0,1], C^1_b(\R^d,\R^d))$, where 
	$C^1_b(\R^d,\R^d)$ denotes the space of $C^1$-mappings which are globally bounded together with its first derivative,  
	has a unique flow 
	\begin{equation}
		\Ph(t,x) = x + \int_0^t u(s,\Ph(s,x))\, ds, \quad x \in \R^d,~ t \in [0,1] 
	\end{equation}
	such that $\Ph(t,\cdot) - \Id \in C^1_b(\R^d,\R^d)$ for all $t$. 
	Given a locally convex space $\cA$ of mappings $f : \R^d \to \R^d$ which is continuously embedded in $C^1_b(\R^d,\R^d)$, 
	one defines the associated \emph{Trouv\'e group} 
	\[
		\cG_\cA := \{\Ph(1,\cdot) : u \in \cF_\cA\},
	\]
	where $\cF_\cA$ is a suitable family of functions $[0,1] \to \cA(\R^d,\R^d)$; 
	this construction is due to Trouv\'e \cite{Trouve95}, see also \cite{Younes10}.  
	We require that the elements of $\cF_\cA$ are \emph{integrable} functions in some sense, 
	but the precise definition depends on the 
	structure of the space $\cA(\R^d,\R^d)$. If the latter space is a Fr\'echet space, then the elements of $\cF_\cA$ are 
	\emph{integrable by seminorm} (see \Cref{seminorm}), in particular, Bochner integrable in the Banach case. 

	There has been a recent interest in a precise description of the regularity properties of the elements of $\cG_\cA$; 
	see \cite{BruverisVialard14} for the Sobolev and \cite{NenningRainer16} for the H\"older case.
	In \cite{Glockner16} similar problems are studied in the context of Banach and Fr\'echet Lie groups.

	In this paper we prove that the Trouv\'e group $\cG_\cA$ is equal to the connected component of the identity 
	$\Diff_0 \cA$
	of 
	\[
		\on{Diff} \cA := \{\Ph \in \Id+\cA(\R^d,\R^d): \inf_{x \in \R^d}\det d\Ph(x) > 0 \}
	\] 
	for all of the following classes $\cA$: 
	\begin{itemize}
   \item Smooth functions with globally bounded derivatives $\cB$ ($=\cD_{L^\infty}$ in \cite{Schwartz66}).
   \item Smooth functions with $p$-integrable derivatives $\Sp$ ($=\cD_{L^p}$ in \cite{Schwartz66}).
   \item Rapidly decreasing Schwartz functions $\cS$.
	 \item Smooth functions with compact support $\cD$.
   \item Global Denjoy--Carleman classes $\BM$.
   \item Sobolev--Denjoy--Carleman classes $\SpM$.
   \item Gelfand--Shilov classes $\SLM$.
	 \item Denjoy--Carleman functions with compact support $\DM$.   
 \end{itemize} 
Hereby, $M=(M_k)$ is a \emph{(strictly) regular} sequence (see \Cref{def:regular}), while 
for the sequence $L=(L_k)$ we just assume $L_k\ge 1$ for all $k$.

In all these cases $\on{Diff} \cA$ is a $C^\infty$ (resp.\ $\CM$, see \Cref{sec:localDC}) regular Lie group, 
by \cite{MichorMumford13} and \cite{KrieglMichorRainer14a}, see also \cite{KrieglMichorRainer16}. 
It thus follows that $\cG_\cA$ has a natural Lie group structure.
In contrast to \cite{MichorMumford13} and \cite{KrieglMichorRainer14a}, where the vector fields were $C^\infty$ 
(resp.\ $\CM$) in time $t$ and hence the $C^\infty$ (resp.\ $\CM$) exponential law applied,   
we are focused in this paper on vector fields which are just integrable in time. 

For the unweighted spaces $\cB$, $W^{\infty,p}$, $\cS$, and $\cD$ we prove that the map which sends a time-dependent 
vector field $u \in L^1([0,1],\cA(\R^d,\R^d))$ to its flow $\Ph_u \in C([0,1],\Diff \cA)$ is continuous.
	
	As an application we obtain that the scale of Bergman spaces on the polystrip with variable width is ODE-closed.
	More precisely:
	The $L^p$-Bergman space $A^p(S_{(r)})$ is the Banach space of $p$-integrable holomorphic functions on the polystrip 
	$S_{(r)} := \{z=(z_1,\ldots,z_d) \in \C^d : \Im (z_i) < r,\, 1 \le i \le d\}$. We prove that the inductive limit 
	\[
		\underrightarrow A^p(\R^d) := \varinjlim_{r>0} A^p(S_{(r)})
	\]
	is topologically isomorphic to the Sobolev--Denjoy--Carleman space $W^{\{\mathbf 1\},p}(\R^d)$ of Roumieu type, where 
	$\mathbf 1 =(1,1,\ldots)$. 
	Similarly, the projective limit 
	\[
		\underleftarrow A^p(\R^d) := \varprojlim_{r>0} A^p(S_{(r)})
	\]
	is topologically isomorphic to the Sobolev--Denjoy--Carleman space $W^{(\mathbf 1),p}(\R^d)$ of Beurling type.
	This allows us to conclude that the flow of any time-dependent vector field $u : [0,1] \to A^p(S_{(r)},\R^d)$, 
	for some $r>0$, which is Bochner integrable in time 
	is a continuous curve $t \mapsto \Ph_u(t)$ in $\Id + \underrightarrow A^p(\R^d,\R^d)$.  
	Analogously, if $u : [0,1] \to \underleftarrow A^p(\R^d,\R^d)$ is integrable by seminorm,  
	then $t \mapsto \Ph_u(t)$ is a continuous curve in $\Id + \underleftarrow A^p(\R^d,\R^d)$.

	The paper is organized as follows. 
	We introduce notation and some necessary background in \Cref{sec:notation}. 
	The function spaces and diffeomorphism groups of interest in this paper are defined in 
	\Cref{spaces}. 
	Our main result $\cG_\cA = \Diff_0 \cA$ is proved in \Cref{sec:diffgroups}.   
	In \Cref{sec:continuity} we show continuity of the flow map in the unweighted cases.
	The application for the scale of Bergman spaces on the polystrip with variable width 
	is given in \Cref{sec:Bergman}.

	\section{Notation and Preliminaries} \label{sec:notation}
	
	In what follows $\N = \{0,1,\dots\}$, $\N_{>0}=\N_{\ge 1} = \N \setminus \{0\}$. 
	For $\al = (\al_1, \dots, \al_d) \in \N^d$, $|\al|:= \al_1 + \cdots +\al_d$ and 
	$\al!:= \al_1!\cdots\al_d!$. By the multinomial theorem, for all $\al \in \N^d$,
	\begin{equation}
		\label{multinomial}
		|\al|!\le d^{|\al|} \al!.
	\end{equation}
	For a $C^k$-map $f: \R^d \to \R^n$ and $|\al|\le k$, $\p^\al f$ denotes the $\al$-th partial 
	derivative of $f$. We also write $\p^\al_x$ to make clear which variable is in the scope of differentiation.

	For a $C^k$-function $f:U\to F$ defined on an open subset $U$ of a Banach space $E$ with values in some 
	other Banach space $F$, we denote by $f^{(k)}=\p^k f = d^k f: U \to L_k(E,F)$ the $k$-th Fr\'echet derivative, 
	where $L_k(E,F)$ is the space of $k$-linear mappings $E^k \to F$ endowed with the operator norm.\par

		For $f \in Z^{X \times Y}$ we consider $f^\vee \in (Z^Y)^X$ defined by $f^\vee(x)(y) = f(x,y)$, and
	with $g \in (Z^Y)^X$ we associate $g^\wedge \in Z^{X \times Y}$ given by $g^\wedge (x,y) = g(x)(y)$.

	\subsection{Sobolev spaces}
		For $k \in \N$ and $p \in [1,\infty]$, let $W^{k,p}(\R^d,\R)$ denote the space of $k$-times 
		weakly differentiable functions defined on $\R^d$ with values in $\R$ such that all 
		partial derivatives up to order $k$ lie in $L^p(\R^d,\R)$. We endow $W^{k,p}(\R^d,\R)$ with the norm 
		\[
			\|f\|_{W^{k,p}}:= \sum_{|\al|\le k} \|\p^\al f\|_{L^p}.
		\]
	An important tool in the subsequent considerations is the following Sobolev inequality (e.g.\ \cite{Adams75}).
	
	\begin{lemma}
		\label{SobolevInequality}
		Let $k:=\lfloor \frac{d}{p} \rfloor +1$. Then $W^{k,p}(\R^d,\R) \subseteq C_b(\R^d,\R)$ (where $C_b$ denotes the space of bounded continuous functions) and there exists a constant $C = C(d,p)$ such that, for all $f \in W^{k,p}(\R^d,\R)$,
		\[
		\|f\|_{L^\infty} \le C\|f\|_{W^{k,p}}.
		\]
	\end{lemma}

	\subsection{Fa\`a di Bruno's formula}
	The Fa\`a di Bruno formula is a generalization of the chain rule to higher order derivatives. The next proposition is a multivariable version, a proof can be found in \cite[Proposition 4.3]{BM04}.
	\begin{proposition}
		\label{fdb}
		Let $f \in C^\infty(\R^m,\R)$ and $g \in C^\infty(\R^n,\R^m)$.
		We have, for all $\ga \in \N^n\setminus \{0\}$, 
		\begin{equation} \label{faa}
			\frac{\p^\ga(f\o g)(x)}{\ga!} = \sum \frac{\al!}{k_1!\cdots k_\ell!}\, \frac{(\p^\al f)(g(x))}{\al!} 
			\Big(\frac{\p^{\de_1}g(x)}{\de_1!}\Big)^{k_1} \cdots \Big(\frac{\p^{\de_\ell}g(x)}{\de_\ell!}\Big)^{k_\ell},
		\end{equation}
		where $\al=k_1+\cdots+ k_\ell$ and the sum is taken over all sets $\{\de_1,\ldots,\de_\ell\}$ of $\ell$ distinct 
		elements of $\N^n \setminus \{0\}$ and all ordered $\ell$-tuples $(k_1,\ldots,k_\ell) \in (\N^m \setminus \{0\})^\ell$, 
		$\ell = 1, 2,\ldots$, such that $\ga = \sum_{i=1}^\ell |k_i| \de_i$.
	\end{proposition}
	
	We also need a version for composition of smooth functions between Banach spaces, cf. \cite[Theorem 1]{Yamanaka89}.

	\begin{proposition}
		\label{banachfdb}
		Let $k \in \N_{\ge 1}$, $E,F,G$ Banach spaces, $U \subseteq E$ and $V \subseteq F$ open subsets, and $f: U \rightarrow F$ and $g: V \rightarrow G$ $C^k$-functions with $f(U) \subseteq V$. Then, for $x \in U$,
		\[
		\frac{(f\circ g)^{(k)}(x)}{k!}= \on{sym_k}\bigg( \sum_{j=1}^k \sum_{\substack{\al \in \N^j_{\ge 1}\\\al_1+\cdots +\al_j =k}} \frac{f^{(j)}(g(x))}{j!} \circ \bigg( \frac{g^{(\al_1)}(x)}{\al_1!} \times \cdots \times \frac{g^{(\al_j)}(x)}{\al_j!}\bigg) \bigg),
		\]
		where $\on{sym_k}$ denotes the symmetrization of $k$-linear mappings.
	\end{proposition}

	\subsection{Integrability by seminorm}
	\label{seminorm}
	Integrability by seminorm is a notion of integration for functions with values in locally convex vector spaces.  
	We briefly recall the necessary definitions and theorems. 
	For a detailed treatment we refer to \cite{Blondia81}.

	Let $E$ be a locally convex vector space and $\cP$ the 
	family of continuous seminorms on $E$. Let $I$ be a bounded interval in $\R$. A function $f:I\rightarrow E$ is called 
	\emph{simple} if it is Lebesgue measurable and only takes finitely many values.  
	It is called \emph{measurable by seminorm} if for each $p \in \cP$, there exists a nullset $N_p \subseteq I$ and a 
	sequence $(f^p_n)_{n \in \N}$ of simple functions such that 
	\[
	p\big(f_n^p(t)-f(t)\big) \overset{n \rightarrow \infty}{\longrightarrow} 0 \quad \text{ for all } t \in I \setminus N_p.
	\]
	It is actually enough to require that property for each $p \in \cP_0$, 
	where $\cP_0$ is a subbase of the continuous seminorms. 
	If $N_p$ and $(f_n^p)$ can be chosen independently from $p$, then $f$ is called \emph{strongly measurable}. 
	Finally, $f$ is called \emph{weakly measurable} if $\ell \circ f: I \rightarrow \R$ is Lebesgue measurable 
	for all $\ell \in E'$. 

	\begin{theorem}[{\cite[Theorem 2.2]{Blondia81}}]
		\label{pettis}
		A function $f:I \rightarrow E$ is measurable by seminorm if and only if $f$ is weakly measurable and for each $p \in \cP$ there exists a nullset $N_p \subseteq I$ such that $f(I \setminus N_p)$ is separable.
	\end{theorem}
	
	For a simple function $f : I \to E$, the integral over a measurable set $J \subseteq I$ is clearly defined as 
	\[
	\int_J f(t)\,dt:= \sum_{y \in E} \la(f^{-1}(\{y\}) \cap J) \cdot y,
	\]
	where $\la$ is the Lebesgue measure on $I$. 
	A function $f : I \to E$ which is measurable by seminorm is called \emph{integrable by seminorm} if
	\begin{gather*}
			\forall p \in \cP,~\forall n \in \N:~p\circ(f_n^p-f) \in L^1(I),\\
			\forall J \subseteq I ~\text{measurable}~\exists F_J \in E~ \forall p \in \cP :~ 
			p\Big(F_J - \int_J f_n^p(t)\,dt\Big) \overset{n \rightarrow \infty}{\longrightarrow} 0;
	\end{gather*}
	in this case $F_J=: \int_J f(t)\,dt$. 
	For a complete locally convex space $E$, integrability of $p\circ f$ 
	for each $p \in \cP$ already implies integrability by seminorm.
	If $E$ is a Banach space, then clearly integrability by seminorm coincides with Bochner integrability and 
	the respective integrals coincide.  
	It is easily seen that for each $p \in \cP$, 
	\begin{equation*}
		p\Big(\int_J f(t)\,dt\Big) \le \int_J p (f(t))\,dt.
	\end{equation*}
	
	Let $\cL^1(I,E)$ consist of functions $f : I \to E$ integrable by seminorm. 
	For a continuous seminorm $p$ on $E$, define $\tilde{p}(f):= \int_I p(f(t)) \,dt$ 
	and let $\tilde{\cP}$ be the family of such seminorms. Then $(\cL^1(I,E),\tilde{\cP})$ 
	is a (non-Hausdorff) locally convex space. 
	For $\cN:= \{f \in \cL^1(I,E): \tilde{p}(f)=0 \text{ for all } \tilde{p} \in \tilde{\cP}\}$, let 
	\[
	L^1(I,E):= \cL^1(I,E) / \cN,
	\]
	which is then clearly also Hausdorff.

	\section{Function spaces and their associated diffeomorphism groups}
	\label{spaces}

	Let us introduce the function spaces considered in this paper.
	We will define spaces $\cA(\R^d,\R)$ of real valued functions for different regularity classes $\cA$, 
	and set $\cA(\R^d,\R^m) := (\cA(\R^d,\R))^m$.
	
	\subsection{Classical spaces of test functions}

		For $p \in [1,\infty]$, we consider 
		\begin{align*}
		W^{\infty,p}(\R^d,\R):&=\bigcap_{k \in \N} W^{k,p}(\R^d)
		=\big\{f \in C^\infty(\R^d,\R): \|f^{(\al)}\|_{L^p}<\infty ~ \forall \al \in \N^d\big\}
		\end{align*}
		and 
		endow it with its natural Fr\'echet topology. 
		For $p = \infty$ we also write $\cB(\R^d,\R):=W^{\infty,\infty}(\R^d,\R)$.
		We consider the Schwartz space 
		\[
		\cS(\R^d,\R) := \big\{f \in C^\infty(\R^d,\R) : \|f\|^{(p,\al)} < \infty \text{ for all } p \in \N, \al \in \N^d \big\},
		\]
		where 
		\[
		\|f\|^{(p,\al)} := \sup_{x \in \R^d}  (1+|x|)^p |f^{(\al)}(x)|,
		\]
		with its natural Fr\'echet topology. 
		We denote by $\cD(\R^d,\R)$ the nuclear (LF)-space of smooth functions on $\R^d$ with compact support.
		
	\subsection{Local Denjoy--Carleman classes}	\label{sec:localDC}
	Let $M=(M_k)$ be a positive sequence. 
	For an open subset $U \subseteq \R^d$ the \emph{Denjoy--Carleman class} of \emph{Beurling} type, 
	denoted by $\CbM(U,\R)$, is the space 
	of $f \in C^\infty(U,\R)$ such that for all compact $K \subseteq U$ and all $\rh>0$ 
	\[
		\|f\|^M_{K,\rh} := \sup_{\al \in \N^d,\, x \in K} \frac{|f^{(\al}(x)|}{\rh^{|\al|}|\al|! M_{|\al|!}} < \infty.
	\] 	
	The Denjoy--Carleman class of \emph{Roumieu} type $\CrM(U,\R)$ is the space of $f \in C^\infty(U,\R)$ 
	such that for all compact $K \subseteq U$ there exists $\rh>0$ with $\|f\|^M_{K,\rh} < \infty$. They are endowed with their 
	natural locally convex topologies.
	We write $\CM$ if we mean either $\CbM$ or $\CrM$.

	Some regularity properties for the sequence $M = (M_k)$ guarantee that the fundamental results of analysis hold true 
	for the classes $\CM$: stability under composition, differentiation, inversion, and solution of ODEs. 
	This applies for sequences which satisfy the following definition. Here we do not strive for the utmost generality, 
	but see \cite{RainerSchindl14} for a characterization of the stability properties.  

	\begin{definition} \label{def:regular}
		We say that a positive sequence $M = (M_k)$ is \emph{regular} if
	 	\begin{enumerate}
			\item $1=M_0 \le M_1 \le M_2 \le \cdots$,
			\item $M$ is log-convex, i.e., $M_k^2 \le M_{k-1}M_{k+1}$ for all $k$, 
			\item $M$ has moderate growth, i.e., there is $C>0$ such that $M_{k+j} \le C^{k+j} M_k M_j$ for all $k,j$. 
		\end{enumerate}
		It is easy to see that for a regular sequence $M$ the Roumieu class $\CrM$ contains the real analytic class $C^\om$.
		In the Beurling case, the inclusion $C^\om \subsetneq \CbM$ is equivalent to 
		\begin{enumerate}
			\item[(4)] $M_k^{1/k}\rightarrow \infty$ as $k \rightarrow \infty$, or equivalently $M_{k+1}/M_k\rightarrow \infty$.
		\end{enumerate}
		A regular sequence $M=(M_k)$ which additionally satisfies (4) is called \emph{strictly regular}.
	 \end{definition} 

	 Evidently, a regular sequence $M=(M_k)$ is \emph{derivation closed}, i.e., there is a $C>0$ such that $M_{k+1} \le C^k M_k$ 
	 for all $k$. We say that a regular sequence $M=(M_k)$ is \emph{quasianalytic} if 
	 \[
	 	\sum_{k=1}^\infty \frac{1}{(k! M_k)^{1/k}} = \infty;
	 \]
	 otherwise it is called \emph{non-quasianalytic}. By the Denjoy--Carleman theorem, $M$ is quasianalytic if and only if 
	 $\CM$ is quasianalytic, i.e., a function $f \in \CM(U)$, where $U$ is a connected open subset of $\R^d$, 
	 is uniquely determined by its Taylor series expansion at any point $a \in U$.

	 For regular sequences $M$ the classes $\CM$ are stable under composition, indeed log-convexity implies the following 
	 inequality (cf.\ \cite[Proposition 4.4]{BM04}) 
	 \begin{equation}
	 	M_1^k \, M_n\ge M_k\, M_{1}^{k_1} \cdots M_{n}^{k_n}, \quad \text{for }k_i\in \N, ~\sum_{i=1}^n i k_i = n, ~\sum_{i=1}^n k_i = k. \label{eq:Childress}
	 \end{equation}
	 For later use we recall the following simple lemma.

	 \begin{lemma}[{\cite[Lemma 2.2]{KrieglMichorRainer14a}}]
		\label{fdbApplication}
		Let $M=(M_k)$ satisfy \eqref{eq:Childress} and let $A>0$. Then there are positive constants $B,C$ depending 
		only on $AM_1$, $m$, and $n$, and $C \to 0$ as $A \to 0$ such that
		\[
		\sum \frac{\al!}{k_1!\cdots k_\ell!} A^{|\al|} M_{|\al|} M_{|\de_1|}^{|k_1|} \cdots M_{|\de_\ell|}^{|k_\ell|} \le B C^{|\ga|} M_{|\ga|}
		\]
		where the sum is as in \Cref{fdb}.
	\end{lemma}

	 We refer to \cite{KMRc}, \cite{KMRq}, \cite{KMRu}, \cite{RainerSchindl14}, and \cite{RainerSchindl16a} 
	for a detailed exposition of the connection between these conditions on $M=(M_k)$
	and the properties of $\CM$.

	For regular sequences $M=(M_k)$, the classes $C^{[M]}$ can be extended to \emph{convenient vector spaces} 
	(i.e.\ Mackey-complete locally convex spaces, cf.\ \cite{KM97}), and they then form cartesian closed categories.
This has been developed in \cite{KMRc}, \cite{KMRq}, and \cite{KMRu}.

    \subsection{Ultradifferentiable spaces of test functions}		
    Let $M=(M_k)$ be a positive sequence. 
    For $p \in [1,\infty]$ and $\si>0$, we consider the Banach space
	\[
		W^{M,p}_\si(\R^d, \R):= \big\{ f \in C^\infty(\R^d, \R): \|f\|^{M,p}_{\si}<\infty \big\},
		\]
		where
		\[
		\|f\|^{M,p}_{\si}:= \sup_{\al \in \N^d} \frac{\|f^{(\al)}\|_{L^p}}{\si^{|\al|} |\al|! M_{|\al|}}.
		\]
		The corresponding \emph{Beurling} class
		\[
		W^{(M),p}(\R^d, \R) := \varprojlim_{n \in \N} W^{M,p}_{1/n}(\R^d,\R),
		\]
		is a Fr\'echet space. The corresponding \emph{Roumieu} class
		\[
		W^{\{M\},p}(\R^d, \R) := \varinjlim_{n \in \N} W^{M,p}_{n}(\R^d,\R),
		\]
		is a compactly regular (LB)-space, 
		see \cite[Lemma 4.9]{KrieglMichorRainer14a}. 
		By writing $W^{[M],p}(\R^d,\R)$ we mean either $W^{(M),p}(\R^d,\R)$ or $W^{\{M\},p}(\R^d,\R)$. 
		For $p=\infty$, we also use $\cB^{[M]}(\R^d,\R):= W^{[M],\infty}(\R^d,\R)$
		and
		\[
			\|f\|^{M}_{\si} := \|f\|^{M,\infty}_{\si}.
		\]

		\begin{remark}
			The definition of $\BM$-mappings makes sense between arbitrary infinite dimensional Banach spaces $E$ and $F$, or even an open subset $U \subseteq E$ of the domain. Then
		\[
		\|f\|^M_{U,\si}:= \sup_{k \in \N, \,x \in U} \frac{\|f^{(k)}(x)\|_{L_k(E,F)}}{\si^k k!M_k},
		\]  
		where $f^{(k)}$ denotes the $k$-th Fr\'echet derivative. The corresponding Beurling and Roumieu classes are defined 
		analogously to the finite dimensional case. In the finite dimensional case both definitions yield the same function 
		spaces and the respective norms are equivalent.
		\end{remark}

		Let $L=(L_k)$ be another positive sequence.
		We consider the Banach space 
		\[
		\cS_{L,\si}^M(\R^d,\R) := \big\{f \in C^\infty(\R^d,\R) : \|f\|_{\si}^{L,M} < \infty\big\}.
		\]
		with the norm
		\[
		\|f\|_{\si}^{L,M} := \sup_{\substack{p \in \N,\, \al \in \N^d\\ x \in \R^d}}  \frac{(1+|x|)^p |f^{(\al)}(x)|}{\si^{p+|\al|}\, p!|\al|!\, L_p M_{|\al|}}.
		\]
		The associated \emph{Gelfand--Shilov} class of Beurling type is the 
		Fr\'echet space 
		\[
		\SbLM(\R^d,\R) := \varprojlim_{n \in \N} \cS^M_{L,1/n}(\R^d,\R).
		\]
		The Gelfand--Shilov class of Roumieu type
		\[
		\SrLM(\R^d,\R) := \varinjlim_{n \in \N} \cS^M_{L,n}(\R^d,\R)
		\]
		is a compactly regular (LB)-space, 
		see \cite[Lemma 4.9]{KrieglMichorRainer14a}. By writing $\SLM(\R^d,\R)$ we mean
		 either $\SbLM(\R^d,\R)$ or $\SrLM(\R^d,\R)$.

		Finally,
		we define 
		\[
		\DM(\R^d,\R) := \CM(\R^d,\R) \cap \cD(\R^d,\R) = \BM(\R^d,\R) \cap \cD(\R^d,\R) 
		\]
		which is non-trivial only if $M=(M_k)$ is non-quasianalytic.
		We equip $\DM(\R^d,\R)$ with the following topology,
		\begin{align*}
		\cD^{[M]}(\R^d,\R) &= \varinjlim_{K \Subset \R^d} \cD^{[M]}_K(\R^d,\R)  
		\end{align*}
		where
		\begin{align*}
		\cD^{\bM}_K(\R^d,\R) :=  \varprojlim_{\ell \in \N} \cD^M_{K, 1/\ell}(\R^d,\R), \quad
		\cD^{\rM}_K(\R^d,\R) :=  \varinjlim_{\ell \in \N} \cD^M_{K,\ell}(\R^d,\R)
		\end{align*}
		and
		\[
		\cD^M_{K,\rh}(\R^d,\R) := \big\{f \in C^\infty(\R^d,\R) : \on{supp} f \subseteq K,\, \|f\|_{\rh}^{M} < \infty\big\}
		\]
		is a Banach space.
		Then $\cD^{\bM}(\R^d)$ is a (LFS)-space and $\cD^{\rM}(\R^d)$ is a Silva space, see \cite{Komatsu73}.

		\begin{remark}
		For $f = (f_1,\ldots,f_m) \in W^{M,p}_{\si}(\R^d,\R^m)$, we also sometimes write $\|f\|^{M,p}_{\si}$ and actually mean 
		$\max_{1\le i \le m} \|f_i\|^{M,p}_{\si}$; similarly, for the other aforementioned function spaces. 
		If domain and codomain are clear from the context, we sometimes omit mentioning them explicitly, 
		e.g., $W^{[M],p}$ instead of $W^{[M],p}(\R^d, \R^m)$.
		\end{remark}

		The next diagram taken from \cite{KrieglMichorRainer14a} describes the mutual inclusion relations 
		of the above spaces.			
		For $1 \le p<q <\infty$, we have the following continuous inclusions: 
			\[
			\xymatrix{
				\cD~ \ar@{{ >}->}[r] & \cS~ \ar@{{ >}->}[r] & \Sp~ \ar@{{ >}->}[r] & W^{\infty,q}~ \ar@{{ >}->}[r] 
				& \cB~ \ar@{{ >}->}[r] & C^\infty \\
				\DrM~ \ar@{{ >}->}[r] \ar@{{ >}->}[u] & \SrLM~ \ar@{{ >}->}[r] \ar@{{ >}->}[u] 
				& \SprM~ \ar@{{ >}->}[r]^{*} \ar@{{ >}->}[u] & W^{\rM,q}~ \ar@{{ >}->}[r]^{*} \ar@{{ >}->}[u] 
				& \BrM~ \ar@{{ >}->}[u] \ar@{{ >}->}[r] & \CrM~ \ar@{{ >}->}[u]\\
				\DbM~ \ar@{{ >}->}[r] \ar@{{ >}->}[u] &\SbLM~ \ar@{{ >}->}[r] \ar@{{ >}->}[u] 
				& \SpbM~ \ar@{{ >}->}[r]^{*} \ar@{{ >}->}[u] 
				& W^{\bM,q}~ \ar@{{ >}->}[r]^{*} \ar@{{ >}->}[u] 
				& \BbM~ \ar@{{ >}->}[u] \ar@{{ >}->}[r] & \CbM~ \ar@{{ >}->}[u]
			}
			\]  
			For the inclusions marked by $*$ we assume that $M=(M_k)$ is derivation closed. 
			If the target is $\R$ (or $\C$) then all spaces are algebras, provided that $(k!M_k)$ is log-convex in the 
			weighted cases,
			and each space in 
			\[
			\xymatrix{
				\cD(\R^d)~ \ar@{{ >}->}[r] & \cS(\R^d)~ \ar@{{ >}->}[r] & \Sp(\R^d)~ \ar@{{ >}->}[r] 
				& W^{\infty,q}(\R^d)~ \ar@{{ >}->}[r] 
				& \cB(\R^d) 
			}
			\]
			is a $\cB(\R^d)$-module, and thus an ideal in each space on its right. Likewise 
			each space in 
			\[
			\xymatrix{
				\DM(\R^d) \ar@{{ >}->}[r] & \SLM(\R^d) \ar@{{ >}->}[r] & \SpM(\R^d) \ar@{{ >}->}[r] 
				& W^{[M],q}(\R^d) \ar@{{ >}->}[r] 
				& \BM(\R^d) 
			}
			\]
			is a $\BM(\R^d)$-module, and thus an ideal in each space on its right.

		\subsection{Associated diffeomorphism groups} \label{sec:diffeo}
		Let $\cA$ be any of the classes $\cB,W^{\infty,p}$, $\cS$, $\cD$, $\BM$, $W^{[M],p}$, $\SLM$, $\DM$. 
		Suppose that $M=(M_k)$ is a regular sequence, resp.\ strictly regular in the Beurling case, 
		and let $L=(L_k)$ be a sequence with $L_k\ge 1$ for all $k$.

		In \cite{KrieglMichorRainer14a} and \cite{MichorMumford13} (see also \cite{KrieglMichorRainer16}) it was shown that
		\begin{align*}
			\on{Diff} \cA:&= \big\{\Ph \in \Id+\cA(\R^d,\R^d): \inf_{x \in \R^d}\det d\Ph(x) > 0 \big\}\\
			&= \big\{\Ph \in \Id+\cA(\R^d,\R^d): \Ph ~\text{is bijective}, ~ \Ph^{-1}\in \Id + \cA \big\}
		\end{align*}
		is a manifold modelled on the open subset $\{\Ph-\Id: \Ph \in \on{Diff}\cA$\} of the convenient vector space 
		$\cA(\R^d,\R^d)$ with global chart $\Ph \mapsto \Ph -\Id$ and actually that it is a smooth (resp. $C^{[M]}$) regular Lie group.
		We have $C^\infty$ injective group homomorphisms 
\[
      \xymatrix{
        \Diff\cD \ar@{{ >}->}[r] & \Diff\cS \ar@{{ >}->}[r] & \Diff W^{\infty,p} \ar@{{ >}->}^{\quad p<q}[r] & \Diff W^{\infty,q} \ar@{{ >}->}[r] & \Diff\cB 
      }
 \]
 		and $\CM$ injective group homomorphisms
 		\[
      \xymatrix{
        \Diff\DM \ar@{{ >}->}[r] & \Diff\SLM \ar@{{ >}->}[r] & \Diff\SpM \ar@{{ >}->}^{~ p<q}[r] & \Diff W^{[M],q} 
        \ar@{{ >}->}[r] & \Diff\BM,
      }
\]
  where each group is a normal subgroup of the groups on its right.

		Our main goal is now to give a different description of $\on{Diff}\cA$ in terms of the so-called Trouv\'e group 
		from image analysis, cf.\ \cite{Trouve95} and \cite{Younes10}. This is outlined in the next section.

	\section{Diffeomorphism groups generated by time-dependent vector-fields}
	\label{sec:diffgroups}

	\subsection{ODE-closedness and the Trouv\'e group.}
		Let $I$ be some interval and $u$ a time-dependent vector field $u:I\times \R^d \rightarrow \R^d$ sufficiently regular, 
		e.g., continuous in $t$ and Lipschitz continuous in $x$ with $t$-integrable global Lipschitz constant, to uniquely solve
		\begin{equation}
		\label{ode}
		x(t) = x_0 + \int_{s_0}^t u(s,x(s))\,ds
		\end{equation}
		for all $s_0,t \in I$ and $x_0 \in \R^d$.\par

		Then one considers, for $t_0 \in I$, the \emph{flow} $\Ph_u(s_0,t_0, x_0) := x(t_0)$ of $u$ at time $t_0$. 
		From now on we assume, unless otherwise stated, that $I = [0,1]$ and $s_0 = 0$ and simply write $\Ph_u(t_0,x_0)$ 
		instead of $\Ph_u(0,t_0,x_0)$. In addition, we set $\ph_u(t_0,x_0):= \Ph_u(t_0,x_0)-x_0$. \par 

		In \cite{Younes10} it was shown that for $u^\vee \in L^1(I, C^1_0(\R^d,\R^d))$ the regularity with respect to 
		the spatial variable is transferred to the flow in the sense that $\Ph_u(t_0, \cdot) \in \Id+C^1_0(\R^d,\R^d)$. 
		In addition, $\Ph_u(t_0,\cdot)^{-1}$ exists and is again an element of $\Id+C^1_0(\R^d,\R^d)$. 
		This yields a way of constructing a multitude of diffeomorphisms of a certain type, 
		namely such of the form $\Id+C^1_0(\R^d,\R^d)$.\par
		
		This way of regularity permanence with respect to ODEs is now generalized to arbitrary locally convex spaces $E$ 
		of maps $\R^d \to \R^d$ and dubbed \emph{ODE-closedness}. To be more precise, let $\cF$ be a family of functions $\cF \subseteq E^I$ containing all constant functions. Then $E$ is called \emph{$\cF$-ODE-closed} if $\ph_{u^\wedge}$ exists and  
		\begin{equation*}
			\ph_{u^\wedge} (t, \cdot) \in E \text{ for all } u \in \cF,~t\in I.
		\end{equation*}
		If $E$ is a Banach space, it is natural to choose $\cF$ to be the family of Bochner integrable functions $I \to E$. 
		For general $E$, among the sensible choices for $\cF$ are functions integrable by seminorm, cf.\ 
		\Cref{seminorm}.\par
		If $\cF$ is a vector space and for $u,v \in \cF$ the functions $w_1, w_2$ defined by 
		\[
			w_1 (t):= 
			\begin{cases}
			u(2t) & \text{ if } t \in [0,1/2]\\
			v(2t-1) & \text{ if } t \in (1/2,1]
			\end{cases}
			\quad \text{ and  } \quad w_2(t):=u(1-t)
		\]
		again belong to $\cF$, 
		then it is not hard to see that
		\[
		\cG_\cF:= \left\{ \Ph_{u^ \wedge}(1,\cdot):u \in \cF \right\}
		\]
		is a group with respect to composition, the so-called \emph{Trouv\'e group}; cf.\ \cite{Trouve95} and \cite{Younes10}. 
		In general not much is known about the Trouv\'e group. 
		Clearly if $E$ is $\cF$-ODE-closed, then $\cG_\cF \subseteq \Id +E$.

		\begin{remark}
		When it does not lead to confusion, we omit writing $(\cdot)^\vee$ and $(\cdot)^\wedge$, e.g., for $u:I\rightarrow E \subseteq (\R^d)^{\R^d}$ we simply write $u(t,x)$ instead of $u^\wedge(t,x)$. Similarly for functions of several variables $f: X_1\times X_2 \rightarrow Y$, by writing $f(x_2)$, we actually mean $f(\cdot,x_2):X_1 \rightarrow Y$. This notational inaccuracy will mainly occur, when writing $\Ph_u(t)$ instead of $\Ph_u(t,\cdot)$.
		\end{remark}

		\subsection{Admissible vector fields}  \label{setup}

		Let $\cA$ be any of the classes $W^{\infty,p}$, $\cS$, $\cD$, $W^{[M],p}$, $\SLM$, $\DM$, where $p \in [1,\infty]$
		(in particular, the cases $\cB$ and $\BM$ are included). 
		Suppose that $M=(M_k)$ is a regular sequence, resp.\ strictly regular in the Beurling case, 
		and let $L=(L_k)$ be a sequence with $L_k\ge 1$ for all $k$. 

		We will show that the Trouv\'e groups for the classes $\cA$ (with suitable choices of $\cF$) 
		coincide with the connected component of the identity in $\on{Diff}\cA$ which we denote by $\on{Diff}_0\cA$. 
		By the results of \cite{KrieglMichorRainer14a} and \cite{MichorMumford13}, see \Cref{sec:diffeo}, 
		we conclude that these Trouv\'e groups have a natural regular Lie group structure. 
		So let us now fix $\cF_{\cA}$ for each $\cA$ mentioned above and  
		write $\cG_{\cA}$ instead of $\cG_{\cF_\cA}$ for the corresponding Trouv\'e group.\par

		For $\cA \in \{W^{\infty,p}, \cS, W^{(M),p}, \cS^{(M)}_{(L)}\}$, 
		when $E:=\cA(\R^d,\R^d)$ is a Fr\'echet space, 
		let $\cF_\cA := L^1(I,E)$ in the sense of \Cref{seminorm}. In particular, this means, for $u \in \cF_\cA$,
		\begin{align}
			\label{infSobolev}
			\forall \al \in \N^d : \int_0^1 \|\p^\al_x u(t)\|_{L^p}\,dt<\infty, &\quad \text{ if }  \cA=W^{\infty,p},
			\\ 
			\label{Schwartz}
			\forall p \in \N,~\forall \al \in \N^d : \int_0^1 \|u(t)\|^{(p,\al)}\,dt <\infty, 
			&\quad \text{ if }  \cA=\cS,
			\\
			\label{ultraBeurlingSobolev}
			\forall \si > 0: \int_0^1 \|u(t)\|^{M,p}_{\si}\,dt < \infty,
			&\quad \text{ if }  \cA=W^{(M),p},
			\\
			\label{BeurlingGelfandShilov}
			\forall \si > 0: \int_0^1 \|u(t)\|^{L,M}_\si \, dt < \infty, 
			&\quad \text{ if }  \cA=\SbLM.
		\end{align}

		In the cases $\cA \in \{ W^{\{M\},p}, \SrML\}$, we take $\cF_\cA$ to be 
		the set of those $u \in L^1(I,E)$ which factor into some step of the inductive limit which represents $E$
		and are (Bochner-)integrable therein. In particular,  
		\begin{align}
		\label{ultraRoumieuSobolev}
			\exists \si > 0: \int_0^1 \|u(t)\|^{M,p}_{\si} \,dt < \infty, &\quad \text{ if }  \cA=W^{\{M\},p},
		\\
		\label{ultraRoumieuGelfand}
			\exists \si > 0: \int_0^1 \|u(t)\|^{L,M}_\si \,dt < \infty, &\quad \text{ if }  \cA=\SrLM.
		\end{align}

		For the compactly supported classes $\cA \in\{\cD, \cD^{(M)}\}$ we take $\cF_\cA$ to be the class of functions that map into some $\cD_K(\R^d,\R^d)$, resp. $\cD^{(M)}_K(\R^d,\R^d)$, and are integrable by seminorm therein. Then 
		there exists a compact subset $K \subseteq \R^d$ such that $\supp u(t) \subseteq K$ for all 
		$t \in [0,1]$, and 
		\begin{align}
			\label{compsupp}
			\forall \al \in \N^d: \int_0^1 \|\p^\al_x u(t)\|_{L^\infty}\,dt < \infty, &\quad \text{ if }  \cA=\cD,
		\\
			\label{bcompsupp}
			\forall \si > 0: \int_0^1 \|u(t)\|^{M}_{\si}\,dt < \infty, &\quad \text{ if }  \cA=\DbM.
		\end{align}		

		Finally in the Roumieu case $\cD^{\{M\}}$, let $\cF_{\cD^{\{M\}}}$ be the class of functions that map into some 
		$\cD^M_{K,\si}(\R^d,\R^d)$ and are Bochner integrable therein. Then
		there exists a compact subset $K \subseteq \R^d$ such that $\supp u(t) \subseteq K$ for all 
		$t \in [0,1]$ and
		\begin{equation}
			\label{rcompsupp}
			 \exists \si >0 :
			\int_0^1 \|u(t)\|^{M}_{\si}\,dt < \infty, \quad \text{ if }  \cA=\DrM.
		\end{equation}
		
		\begin{remark}
			\label{measurable}
			Since point evaluation is continuous on all spaces $\cA(\R^d,\R^d)$, we get, by  
			\Cref{pettis}, that the function
			\[
			I \rightarrow \R^d,~s \mapsto u(s,x), 
			\] 
			is measurable for all $x \in \R^d$. Then it is straightforward to check that for all $\ph \in C(I,\R^d)$ and $x \in \R^d$, the mapping
			\[
			I \rightarrow \R^d, ~ s \mapsto u(s, x+\ph(s)),
			\]
			is measurable; see \cite[Lemma 2.2]{AulbachWanner96} for a detailed argument.
		\end{remark}

		\subsection{Main result}
		
		\begin{theorem}
			\label{maintheorem1}
			For all classes $\cA$ and corresponding $\cG_\cA$ introduced in \Cref{setup}, 
			\[
			\cG_\cA = \on{Diff}_0\cA.
			\]
			For each time-dependent vector field $u \in \cF_\cA$ the flow $t \mapsto \ph_u(t) \in \cA$ 
			is continuous.
 		\end{theorem}
 		
		The proof of \Cref{maintheorem1} is subdivided into several propositions. 
		One inclusion can be proved uniformly for all classes $\cA$. 
		Abusing notation we shall denote by $\cA$ also the space $\cA(\R^d,\R^d)$.
		
		\begin{proposition}
			\label{maintheorem}
			$\cG_\cA \supseteq \on{Diff}_0\cA$.
		\end{proposition}
		
		\begin{proof}
			Observe that $\on{Diff}\cA-\Id$ is open in $\cA$, and hence so is $\on{Diff}_0\cA-\Id$. 
			Since $\on{Diff}_0\cA-\Id$ is connected and locally path-connected, it is path-connected. 
			Thus any $\Ph \in \on{Diff}_0\cA$ can be connected by a polygon with the identity.

			Now take $\Ph = \Id +\ph \in \on{Diff}_0\cA$ such that $\ga(t):= \Id + t \ph \in \on{Diff}\cA$ for 
			$0\le t \le 1$. Since $\on{Diff}\cA$ is a Lie group, $u(t):= \ph\circ\ga(t)^{-1}$ is smooth. 
			Hence in all cases $u$ is in $\cF_\cA$ (in the Roumieu cases a smooth curve factors through a 
			step in the defining inductive limit, since it is compactly regular or Silva).
			So $\Ph = \Ph_u(1,\cdot) \in \cG_\cA$.
 
			For the general case, let $\Ph \in \on{Diff}_0\cA$. Take a polygon with vertices 
			$\Id=\Ph_1, \dots, \Ph_n=\Ph$ connecting $\Id$ with $\Ph$. 
			By the previous paragraph, we know already that $\Ph_2 \in \cG_\cA$. 
			Now we may argue iteratively. 
			By assumption, the segment between $\Ph_2$ and $\Ph_3$ lies in $\on{Diff}\cA$ and therefore also 
			the segment between $\Id$ and $\Ph_3\circ \Ph_2^{-1}$ which shows that $\Ph_3\circ \Ph_2^{-1} \in \cG_\cA$. 
			Since $\cG_\cA$ is a group, we may conclude that $\Ph_3 \in \cG_\cA$. 
			Applying the same argument for the remaining vertices finally shows that $\Ph \in \cG_\cA$.
		\end{proof}

		The other inclusion $\cG_\cA \subseteq \on{Diff}_0\cA$ follows from the following assertions:
		\begin{itemize}
			\item[($\bullet$)] \emph{For each time-dependent vector field $u \in \cF_\cA$ the flow $t \mapsto \ph_u(t) \in \cA$ 
					is continuous.}
			\item[($\bullet \bullet$)]	\emph{$\inf_{x \in \R^d} \det d\Ph_u(t,x) > 0$ for all  $t \in I$}.	\label{detpositive}
		\end{itemize} 
		Let us check ($\bullet \bullet$). 
		Observe that there is some $v \in \cF_\cA$ such that $\Ph_u(t)^{-1} = \Ph_v(t)$ for all $t$.  
		In particular, $\Ph_u(t)^{-1}$ is continuously differentiable on $\R^d$ and 
		the first derivative is invertible at any point $x \in \R^d$. 
		For fixed $x \in \R^d$, the mapping $t \mapsto \ga(t):= \det d\Ph_u(t,x)$ is continuous and has values in 
		$\R \setminus \{0\}$. 
		Since $\ga(0)=1$, it thus follows that $\ga(t) = \det d\Ph_u(t,x) > 0$ for all $x \in \R^d$ and $t \in I$. 
		Now suppose there exists a sequence $x_n$ in $\R^d$ such that $\det d\Ph_u(t,x_n) \rightarrow 0$ 
		as $n \rightarrow \infty$. This would imply  
		the existence of a sequence $y_n$ such that $\det d \Ph_v (t,y_n) \rightarrow \infty$. 
		But this contradicts ($\bullet$), since all spaces $\cA$ are contained in $\cB$.

		The rest of the section is devoted to prove ($\bullet$).
		Depending on $\cA$, the techniques will be different.

		\subsection{Ultradifferentiable classes}
		First we treat the weighted Roumieu-type spaces.
		
		\begin{proposition}
			\label{MinfRoumieu}
			Let $M=(M_k)$ be a regular sequence. 
			For each time-dependent vector field $u \in \cF_{\cB^{\{M\}}}$ the flow $t \mapsto \ph_u(t) \in \cB^{\{M\}}$ 
			is continuous.
		\end{proposition}
		
		\begin{proof}
			Take $u \in \cF_{\cB^{\{M\}}}$. Due to \eqref{ultraRoumieuSobolev}, 
			there exists $\rh>0$ such that $\int_0^1 \|u(t)\|^M_{\rh} \,dt < \infty$. 
			We show first that $t \mapsto \Ph_{u}(t, \cdot)$ has values in $\Id + \cB^{\{M\}}$ and 
			later on continuity in $t$.

			Let $\de > 0, ~t_0 \in I$ and set $J := I \cap [t_0-\de, t_0+\de]$. 
			For $\ph \in C(J, \R^d)$ and $x \in \R^d$, define
			a mapping $T: C(J,\R^d)\times\R^d \rightarrow C(J,\R^d)$ by setting
			\begin{equation*}
				T(\ph,x)(t):= \int_{t_0}^t u(s, x+\ph(s))\,ds.
			\end{equation*}
			We claim that 
			$T(\cdot,x)$ is Lipschitz with Lipschitz constant less than $1$ if $\de$ is chosen sufficiently small. For,
			\begin{align*}
			|T(\ph,x)(t) - T(\ps,x)(t)| & = \big| \int_{t_0}^t u(s,x+\ph(s)) - u(s, x+\ps(s)) \,ds \big|\\
			&\le \int_{t_0}^t |u(s,x+\ph(s)) - u(s, x+\ps(s))|\,ds\\
			&\le  \rh M_1 \int_{t_0}^t \|u(s)\|^M_{\rh} |\ph(s)-\ps(s)| \,ds\\
			&\le \rh M_1 \int_J \|u(s)\|^M_{\rh} \,ds ~ \|\ph-\ps\|_{L^\infty},
			\end{align*}
			and $\int_J \|u(s)\|^M_{\rh} \,ds$ gets arbitrarily small (independent of $t_0$) by choosing $\de$ small enough. So let $\de$ be chosen such that 
			\begin{equation}
			\label{deltachoice}
			\max(1,\rh) M_1 \int_J \|u(s)\|^M_{\rh} \,ds \le 1/2.
			\end{equation}
			Then, for each fixed $x$,
			 $T(\cdot,x)$ is a contraction on $C(J,\R^d)$ and has a unique fixed point $\ph(x) \in C(J, \R^d)$, i.e., 
			 for all $|t-t_0|\le \de$,
			\begin{equation} \label{fixedpoint}
				\ph(x)(t) = T(\ph(x),x)(t) = \int_{t_0}^t u(s, x+\ph(x)(s))\,ds.	
			\end{equation}

			Patching the solutions together appropriately (this will be carried out in more detail later, 
			see \Cref{gluing}), we get a solution of \eqref{ode} on $I$. 
			Thus we can define, for each $t \in I$, $\Ph_u(t)=\Id+\ph_u(t)$ and we have to show that 
			\[
			\ph_u(t) \in \cB^{\{M\}}(\R^d,\R^d) \quad \text{ for all } t \in I.
			\]
			To this end we need some additional observations.\par 
			
			\begin{claim} \label{claim:1}
				$T \in C^\infty(C(J,\R^d)\times\R^d,C(J,\R^d))$ and for $k \ge 0$ and $\al \in \N^d$, we get
			\begin{equation}
			\label{derivative}
			\p^k_\ph \p^\al_x T(\ph,x)(\ps_1, \dots, \ps_k)(t) 
			= \int_{t_0}^t d^k \p_x^\al u(s, x+\ph(s))(\ps_1(s), \dots, \ps_k(s))\,ds,
			\end{equation}
			where $\p^k_\ph$ denotes the $k$-th (Fr\'echet-)derivative with respect to $\ph \in C(J, \R^d)$,  
			$\p^\al_x$ the $\al$-th partial derivative with respect to $x \in \R^d$, 
			$d^k$ the $k$-th total derivative with respect to the second variable (in $\R^d$).
			\end{claim}

			We illustrate the inductive start, the induction step works along similar lines. For the existence (and continuity) of the first derivative it is enough to show existence and continuity of the first partial derivatives. Fix $x \in \R^d$ and $\ph, \ps \in C(J,\R^d)$. Then, for all $t \in J$,  
			\begin{align*}
			\MoveEqLeft
			|T(\ph+\ps,x)(t) - T(\ph,x)(t) - \int_{t_0}^td u(s, x+\ph(s))\cdot\ps(s) \,ds| \\
			\le~ & \int_{t_0}^t| u(s, x+\ph(s) + \ps(s)) - u(s, x+\ph(s)) - d u(s, x+\ph(s))\cdot\ps(s)| \,ds \\
			\le ~& \int_{t_0}^t \sup_{y \in \R^d}\|  d^2 u(s, y)\|_{L_2(\R^d, \R^d)}|\ps(s)|^2 \,ds\\
			\le ~& 2 \rh^2 M_2\|\ps\|_{L^\infty}^2 \int_{0}^1 \|u(s)\|^M_{\rh}\,ds.
			\end{align*} 
			Therefore $\p_\ph T(\ph,x)(\ps) (t)= \int_{t_0}^t d u(s, x+\ph(s))\cdot\ps(s) \,ds$, which is easily seen to be 
			continuous as a mapping from $C(J,\R^d)\times \R^d \rightarrow L(C(J,\R^d),C(J,\R^d))$. The existence and continuity 
			of the derivative with respect to $x$ can be proven similarly. This implies \eqref{derivative} for $k = 1$ and 
			$|\al| = 1$. Now proceed by induction to finish the proof of \Cref{claim:1}.\par

			\begin{claim} \label{claim:2}
				$T \in \cB^{\{M\}}(C(J,\R^d)\times \R^d, C(J,\R^d))$.				
			\end{claim}

			Using \eqref{derivative}, we get for $\|\ps_i\|_{L^\infty} \le 1$, 
			\begin{align*}
			\|\p^k_\ph \p^\al_x T (\ph,x)(\ps_1, \dots,\ps_k)\|_{L^\infty} &\le \rh^{k+|\al|}(k+|\al|)!M_{k+|\al|}  
			\int_J \|u(s)\|^M_{\rh}\,ds
			\end{align*}
			which gives $\|T\|^M_{C(J,\R^d)\times\R^d,\rh} \le 1/2$, by \eqref{deltachoice}, and thus proves \Cref{claim:2}. 
			In particular,
			\begin{equation}
			\label{Tderivative}
			\|\p_\ph T(\ph,x)\|_{L(C(J,\R^d), C(J,\R^d))} \le 1/2, \quad  \ph \in C(J,\R^d),~ x \in \R^d.
			\end{equation}

			\medskip

			Consider the mapping 
			\[
			S: C(J,\R^d)\times \R^d \rightarrow C(J,\R^d), \quad (\ph,x) \mapsto \ph - T(\ph,x).
			\]
			Then $\partial_\ph S(\ph,x) = \Id - \partial_\ph T(\ph,x) \in L(C(J,\R^d), C(J,\R^d))$ admits a bounded inverse due to \eqref{Tderivative}. The inverse is given by the Neumann series $\sum_{k = 0}^\infty \partial_\ph T(\ph,x)^k$, which thus gives 
			\[
			\|\partial_\ph S(\ph,x)^{-1}\|_{L(C(J,\R^d), C(J, \R^d))} \le 2.
			\]
			For all $(k,\al) \in \N^{1+d}$ except $(0, 0, \dots,0)$ and $(1,0,\dots,0)$, we get
			\[
			\|\partial_\ph^k \partial_x^\al S(\ph,x)\|\le \rh^{k+|\al|}M_{k+|\al|}(k+|\al|)!\, \|T\|^M_{C(J,\R^d)\times\R^d,\rh} ,
			\]
			and for $(1,0,\dots,0)$ we clearly have
			\[
			\|\partial_\ph S(\ph,x)\| \le 1+ \rh M_1 \|T\|^M_{C(J,\R^d)\times\R^d,\rh}.
			\]
			Since $M$ is a regular sequence, we can apply the implicit function theorem for the class $\cB^{\{M\}}$, 
			cf.\ \cite[Theorem 3]{Yamanaka89}.  
			Thus, $x \mapsto \ph(x)$ fulfills the $\cB^{\{M\}}$-estimates for derivatives of order $\ge 1$. 
			Moreover, \eqref{fixedpoint} and $\int_0^1\|u(t)\|_{L^\infty} \,dt < \infty$ imply that $x\mapsto \ph(x)$ is globally bounded. So there exists (for each $J$) a $\ta>0$ such that $\ph \in \cB^{M}_\ta(\R^d, C(J,\R^d))$. Since we can cover $[0,1]$ by finitely many $J$, we can choose one $\ta$ valid for all $J$. Observe however that $\ph$ (the fixed point of $T$) depends on the interval $J$.\par 

			\begin{claim} \label{gluing}
 				There is a global solution $\ph$, 
			\begin{equation}
			\label{phisolution}
			\ph(x)(t) = \int_0^t u(s,x+\ph(x)(s))\,ds, \quad  t \in I,
			\end{equation}
			and there exists $\la>0$ such that $\ph(\cdot)(t) \in \cB^M_\la(\R^d,\R^d)$ and it is uniformly bounded therein with respect to $t \in I$.	
			\end{claim}

			At this point, we know that, for each $J_k:= [(k-1)\de, k\de]$, 
			there exists $\ph_k \in \cB^{M}_\ta (\R^d, C(J_k, \R^d))$ such that
			\[
			\ph_k(x)(t) = \int_{(k-1)\de}^t u(s,x+\ph_k(x)(s))\,ds, \quad  t \in J_k,\,  x \in \R^d.
			\]
			Since $\ev_t$ is a bounded linear operator on $C(J,\R^d)$, we have 
			$\p^l(\ev_t \circ \ph_k)(x) =  \ev_t \p^l\ph_k(x)$, 
			and therefore $x \mapsto \ph_k(x)(t) = \ev_t \circ \ph (x)\in \cB^{M}_\ta (\R^d,\R^d)$ for any $t \in J_k$. 
			Now we can define $\ph$ iteratively: 
			For $t \in J_1$ and $x \in \R^d$, set $\ph(x)(t):= \ph_1(x)(t)$. 
			Suppose we have already defined $\ph$ on $[0, k\de]$. Then set 
			$y_k(x):= x+\ph(x)(k\de) \in \Id+\cB^{\{M\}}(\R^d,\R^d)$. And for $t \in J_{k+1}$, let 
			\[
			\ph(x)(t):= \ph_{k+1}(y_k(x))(t).
			\]
			Using composition-closedness of $\cB^{\{M\}}$ 
			(cf.\ \cite[Theorem 6.1]{KrieglMichorRainer14a}), 
			it is clear that $\ph(\cdot)(t) \in \cB^{\{M\}}(\R^d, \R^d)$ for $t \in [0,(k+1)\de]$. 
			Iterating this procedure sufficiently many times, 
			finally gives \eqref{phisolution} and the construction also yields the additional boundedness condition.
			\Cref{gluing} is proved.

			We are left to prove continuity of $t \mapsto \Ph_{u}(t)$ as a mapping into $\Id + \cB^{\{M\}}$. 

			\begin{claim} \label{claim:cont}
			There exists $\si \ge \rh$ such that $t \mapsto \ph_u(t)$ is continuous into $\cB^M_\si(\R^d,\R^d)$.				
			\end{claim}
 
			By \Cref{gluing}, there exists $\la > 0$ such that $\ph_{u}(t)$ is  bounded (uniformly in $t$) in 
			$\cB^{M}_\la(\R^d,\R^d)$ by some constant $C$. And we have for $0\le r\le t\le 1$, $\ga \in \N^d$ and $x \in \R^d$, 
			\begin{equation}
			\label{continuityinequ}
			\frac{|\p^\ga_x \ph_u(t,x) -  \p^\ga_x \ph_u(r,x)|}{\si^{|\ga|}|\ga|! M_{|\ga|}} \le \int_r^t  \frac{|\p^\ga_x (u (s,\Ph_u(s,x)))|}{\si^{|\ga|}|\ga|! M_{|\ga|}}\,ds,
			\end{equation}
			where we used \eqref{ultraRoumieuSobolev} to justify interchanging differentiation and integration. Next we apply 
		Fa\`a di Bruno's formula \eqref{faa} to the integrand and get
			\begin{align*}
				\MoveEqLeft
				\frac{\|\p^\ga_x(u(s)\circ \Ph_u(s))\|_{L^\infty(\R^d)}}{\ga! M_{|\ga|}} 
				\le \sum \frac{\al!}{k_1! \dots k_l!}\frac{\|\p^\al_x u(s)\|_{L^\infty(\R^d)}}{\al! M_{|\al|}}\\
				&\quad\quad\quad\quad\quad \cdot\bigg(\frac{\|\p^{\de_1}_x \Ph_u(s)\|_{L^\infty(\R^d)}}{\de_1! M_{|\de_1|}}\bigg)^{|k_1|}\cdots  \bigg(\frac{\|\p^{\de_l}_x \Ph_u(s)\|_{L^\infty(\R^d)}}{\de_l! M_{|\de_l|}}\bigg)^{|k_l|}\\
				&\quad\quad\quad\quad \le \sum \frac{\al!}{k_1! \dots k_l!}(d\rh (C+1))^{|\al|}(d\la)^{|\ga|}\frac{\|\p^\al_x u(s)\|_{L^\infty(\R^d)}}{\rh^{|\al|}|\al|! M_{|\al|}}\\
				&\quad\quad\quad\quad \le \sup_{\be \in \N^d} \frac{\|\p^\be_x u(s)\|_{L^\infty(\R^d)}}{\rh^{|\be|}|\be|! M_{|\be|}} (d\la)^{|\ga|}\sum \frac{\al!}{k_1! \dots k_l!}(d\rh (C+1))^{|\al|}.
			\end{align*}
			where the summation is as in \eqref{faa}. For the first inequality we used \eqref{eq:Childress}, 
			the second one follows from the boundedness condition and \eqref{multinomial}. 
			The sum on the right-hand side is $\le D \ta^{|\ga|}$, for some $D,\ta>0$, by \Cref{fdbApplication}.  
			Choosing $\si = d\la\ta$ and plugging this into the right-hand side of \eqref{continuityinequ} implies
			\Cref{claim:cont}.

			The proof of the proposition is complete.
		\end{proof}
		
		\begin{proposition}
			\label{MpRoumieu}
			Let $M=(M_k)$ be a regular sequence and $p \in [1,\infty)$. 
			For each time-dependent vector field $u \in \cF_{W^{\{M\},p}}$ the flow $t \mapsto \ph_u(t) \in W^{\{M\},p}$ 
			is continuous.
		\end{proposition}
		
		\begin{proof}
			Let $u \in \cF_{W^{\{M\},p}}$, i.e., there exists $\rh > 0$ such that $\int_0^1 \|u(t)\|^{M,p}_{\rh} \,dt < \infty$. Then $u \in \cF_{\cB^{\{M\}}}$, by \Cref{SobolevInequality},  
			and thus by \Cref{MinfRoumieu}, 
			$I \ni t \mapsto \Ph_u(t)$ is a curve of diffeomorphisms in $\on{Diff}\BrM$. Therefore, there is $B>0$ such that
			\begin{align}
			\label{integralest}
			\begin{split}
				\int_{\R^d} |\p^\al_x u(s,\Ph_u(s,x))|^p\,dx&=\int_{\R^d} \frac{|\p^\al_x u(s,y)|^p}{|\det d \Ph_u(s,\Ph_u^{-1}(s,y))|}\,dy\\
				&\le B \int_{\R^d} |\p^\al_x u(s,y)|^p \,dy.
				\end{split}
			\end{align} 
			Using this, we may argue analogously as in the end of the proof of \Cref{MinfRoumieu}. First we apply Minkowski's integral inequality to get an analogue of \eqref{continuityinequ}. 
			For $0\le r \le t\le 1$ and $\si \ge \rh$,
			\begin{align}
			\label{continuityinequ2}
				\sup_{\ga} \frac{\|\p^\ga_x\ph_u(t)-\p^\ga_x\ph_u(r)\|_{L^p(\R^d)}}{\si^\ga |\ga|!M_{|\ga|}}\le \int_r^t \sup_{\ga} \frac{\|\p^\ga_x( u(s)\circ\Ph_u(s))\|_{L^p(\R^d)}}{\si^\ga |\ga|! M_{|\ga|}} \,ds.
			\end{align}
			For the integrand on the right-hand side, we use the Fa\`a di Bruno formula \eqref{faa} and \eqref{integralest} to get
			\begin{align*}
				\MoveEqLeft
				\frac{\|\p^\ga_x (u(s)\circ\Ph_u(s))\|_{L^p(\R^d)}}{\ga! M_{|\ga|}} 
				\le  B^{1/p} \sum \frac{\al!}{k_1! \dots k_l!}\frac{\|\p^\al_x u(s)\|_{L^p(\R^d)}}{\al! M_{|\al|}} \\
				&\hspace{3.5cm} \cdot\bigg(\frac{\|\p^{\de_1}_x \Ph_u(s)\|_{L^\infty(\R^d)}}{\de_1! M_{|\de_1|}}\bigg)^{|k_1|}\cdots  \bigg(\frac{\|\p^{\de_l}_x \Ph_u(s)\|_{L^\infty(\R^d)}}{\de_l! M_{|\de_l|}}\bigg)^{|k_l|},
			\end{align*}
			and then complete the proof as in \Cref{MinfRoumieu}.		
		\end{proof}

		\begin{proposition}
			\label{gelfandshilovroumieuprop}
			Let $M=(M_k)$ be a regular sequence and $L = (L_k)$ satisfy $L_k \ge 1$. 
			For each time-dependent vector field $u \in \cF_{\SrLM}$ the flow $t \mapsto \ph_u(t) \in \SrLM$ 
			is continuous.
		\end{proposition}
		
		\begin{proof}
			Let $u \in \cF_{\cS^{\{M\}}_{\{L\}}}$ and let $\si$ be as in \eqref{ultraRoumieuGelfand}. 
			Then clearly $u \in \cF_{\cB^{\{M\}}}$, it maps into the step $\cB^{M}_\si$, and is integrable therein. 
			Due to \Cref{MinfRoumieu}, we thus know that $t \mapsto \Ph_u(t)$ is continuous into $\on{Diff}\cB^{\{M\}}$. 
			In particular, there is $\rh>0$ and $C_1>0$ such that
			\begin{equation} \label{eq:BMest}
				\sup_{t \in I} \|\ph_u(t)\|^M_{\rh} \le C_1	
			\end{equation} 
			For $0\le r \le t \le 1$, we have
			\begin{align*}
				\frac{(1+|x|)^p|\p^\ga_x\ph_u(t,x)-\p^\ga_x\ph_u(r,x)|}{p!|\ga|! L_p M_{|\ga|}}
				\le \int_r^t \frac{(1+|x|)^p|\p^\ga_x u(s)\circ\Ph_u(s,x)|}{p!|\ga|! L_p M_{|\ga|}}\,ds.
			\end{align*}
			By Fa\`a di Bruno's formula \eqref{faa} and \eqref{eq:BMest} (replacing $C_1$ by $1+ C_1$),  
			\begin{align*}
				\MoveEqLeft
				\frac{|\p^\ga_x (u(s)\circ\Ph_u)(s,x)|}{|\ga|!  M_{|\ga|}}
				\le (d\rh)^{|\ga|} \sum \frac{\al!}{k_1! \cdots k_l!} (dC_1)^{|\al|} 
				\frac{|\p^\al_x u(s)(\Ph_u(s,x))|}{|\al|! M_{|\al|}}.	
			\end{align*}	
			By \eqref{eq:BMest}, $(1+|x|)(1+|\Ph_u(s,x)|)^{-1}$ is uniformly bounded, say by $C_2$, in $s$ and $x$. 
			Thus,
			\begin{align*}
			 	\frac{(1+|x|)^p|\p^\al_x u(s)(\Ph_u(s,x))|}{p!|\al|! L_p M_{|\al|}} 
			 	&\le 
			 	C_2^p  \frac{(1+|\Ph_u(s,x)|)^p|\p^\al_x u(s)(\Ph_u(s,x))|}{p!|\al|! L_p M_{|\al|}} 
			 	\\
			 	& \le 
			 	C_2^p \si^{p+|\al|} \|u(s)\|^{L,M}_\si. 
			 \end{align*} 
			By \Cref{fdbApplication}, there are $C_3,C_4>0$ such that 
			$\sum \frac{\al!}{k_1! \cdots k_l!} (dC_1 \si)^{|\al|} \le C_3 C_4^{|\ga|}$.		
			For $\ta=\max\{C_2\si, C_4d\rh\}$ we hence obtain
			\begin{equation*}
				\|\ph_u(t,\cdot)-\ph_u(r,\cdot)\|^{L,M}_\ta \le C_3 \int_r^t \|u(s)\|^{L,M}_\si \,ds,
			\end{equation*}
			which shows continuity of $t \mapsto \ph_u(t) \in \SrLM$ since $s\mapsto \|u(s)\|^{L,M}_\si$ is integrable.
		\end{proof}
		
		\begin{proposition}
			\label{Roumieucomp}
			Let $M=(M_k)$ be a regular non-quasianalytic sequence. 
			For each time-dependent vector field $u \in \cF_{\DrM}$ the flow $t \mapsto \ph_u(t) \in \DrM$ 
			is continuous.
		\end{proposition}
		
		\begin{proof}
			There is $R>0$ such that $\bigcup_{t \in I}\supp u(t)$ is contained in the ball $B_R(0)$ with center $0 \in \R^d$
			and radius $R$.	
			It is easily seen that 
			$\bigcup_{t \in I} \supp \ph_u(t) \subseteq B_{R+M}(0)$, where $M=\sup_{t\in I}\|\ph_u(t)\|_{L^\infty(\R^d)}$. The rest follows from \Cref{MinfRoumieu}.
		\end{proof}
		
		Next we treat the Beurling analogues of the above spaces.
		
		\begin{proposition}
			\label{BeurlingSobolev}
			Let $M=(M_k)$ be a regular sequence and $p \in [1,\infty]$. 
			For each time-dependent vector field $u \in \cF_{W^{(M),p}}$ the flow $t \mapsto \ph_u(t) \in W^{(M),p}$ 
			is continuous.
		\end{proposition}
		
		\begin{proof}
			First let $p = \infty$. 
			Take $u \in \cF_{\cB^{(M)}}$ and fix some $\rh_0>0$. Then $u$ is integrable in $\cB^M_{\rh_0}$. 
			The proof of \Cref{MinfRoumieu} (see \Cref{gluing}) implies that 
			there are $C,\la>0$ depending on $\rh_0$ such that 
			$\sup_{t \in I}\|\vh_u(t)\|^M_\la \le C$. 
			In the estimate of the integrand in \eqref{continuityinequ} now take any $\rh<\rh_0$ 
			and observe that the $\ta$ tends to $0$ 
			as $\rh$ tends to $0$, by \Cref{fdbApplication}. 
			Therefore the $\si$ may be chosen arbitrarily small, which yields continuity as 
			a map into $\cB^{(M)}$. 
			The case $p < \infty$ works analogously.
		\end{proof}

		\begin{proposition}
			\label{GSB}
			Let $M=(M_k)$ be a regular sequence and let $L = (L_k)$ satisfy $L_k \ge 1$. 
			For each time-dependent vector field $u \in \cF_{\SbLM}$ the flow $t \mapsto \ph_u(t) \in \SbLM$ 
			is continuous.
		\end{proposition}
		
		\begin{proof}
			Just repeat the proof of \Cref{gelfandshilovroumieuprop} and observe that for any choice of small $\si>0$ in the beginning of the proof, $\rh$ and $C_1$ can be chosen uniformly, and thus the $\ta$ in the end of the proof gets small as well.
		\end{proof}
		
		\begin{proposition}
			\label{DB}
			Let $M=(M_k)$ be a regular non-quasianalytic sequence. 
			For each time-dependent vector field $u \in \cF_{\DbM}$ the flow $t \mapsto \ph_u(t) \in \DbM$ 
			is continuous.
		\end{proposition}
		
		\begin{proof}
			That $\bigcup_{t\in I}\supp \ph_u(t)$ is bounded follows as in \Cref{Roumieucomp}. The rest follows from  
			\Cref{BeurlingSobolev}.
		\end{proof}

				\begin{remark} \label{rem:Beurling}
					Let us sketch an alternative proof of \Cref{BeurlingSobolev}, \Cref{GSB}, and \Cref{DB} 
					for \emph{strictly} regular $M$:
					First consider the case $\BbM$. 
					Let $u \in \cF_{\BbM}$ and set 
					\[
						L_k := \frac{1}{k!} \sum_{|\al| = k} \int_0^1 \|\p_x^\al u(t,\cdot)\|_{L^\infty}\, dt.
					\] 
					By \cite[Lemma 6]{Komatsu79b}, there is a strictly regular sequence $N \ge L$ such that 
					$(N_k/M_k)^{1/k} \to 0$. Hence, \Cref{MinfRoumieu} implies that 
					$t \mapsto \ph_u(t)$ is a continuous map $I \to \cB^{\{N\}}$, and thus 
					a continuous map $I \to \cB^{(M)}$, since $(N_k/M_k)^{1/k} \to 0$ entails that    
					$\cB^{\{N\}}$ is continuously included in $\cB^{(M)}$.

					If $p<\infty$ and $u \in \cF_{W^{(M),p}}$, then 
					$t \mapsto \ph_u(t)$ is a continuous map $I \to \cB^{(M)}$, by the previous paragraph. 
					That $t \mapsto \ph_u(t)$ actually has values in $W^{(M),p}$ follows easily from 
					\[
						\ph_u(t) = \int_0^t u(s) \o \Ph_u(s) \, ds,   	
					\]
					and its continuity as map $I \to W^{(M),p}$ is shown similarly, since for $t\ge r$,
					\[
						\ph_u(t) - \ph_u(r)  = \int_r^t u(s) \o \Ph_u(s) \, ds.   	
					\]

					The case $\SbLM$ is treated analogously and for $\DbM$ only a condition for the support 
					has to be checked.
				\end{remark}
		
		\subsection{Unweighted classes}

		\begin{proposition} \label{unweighted}
			Let $\cA$ be any of the classes $W^{\infty,p}$, for $1 \le p \le \infty$, $\cS$, and $\cD$. 
			For each time-dependent vector field $u \in \cF_{\cA}$ the flow $t \mapsto \ph_u(t) \in \cA$ 
			is continuous.
		\end{proposition}

		\begin{proof}
			We may assume that there exists a unique continuous map $\ph_u$ such that
			\begin{equation} \label{sol}
			\ph_u(t,x) = \int_0^t u(s,x+\ph_u(s,x))\, ds, \quad x \in \R^d,~ t \in [0,1]. 
			\end{equation}
			Let us consider the case $\cA=\cB$. 
			That $\ph_u(t)$ is uniformly bounded with respect to $t$ in $\cB(\R^d,\R^d)$ is the content of 
			\cite[Theorem 8.9]{Younes10}, 
			at least for vector fields $u$ vanishing at infinity together with all derivatives. 
			But the proofs can be adjusted to work for the larger class of time-dependent $\cB$-vector fields. 
			Nevertheless we recall a proof of the uniform boundedness.

			It is clear, from \eqref{sol} and $\int_0^1 \|u(s)\|_{L^\infty} \, ds < \infty$, that 
			$\ph_u(t)$ is globally bounded, uniformly for all $t$.
			In order to show that $\p_x^k\vh_u(t)$ is globally bounded, uniformly in $t$, for each $k$, 
			we use induction on $k$. So suppose that we already know that $\p_x^h \ph_u(t)$ is globally bounded for each $t$ 
			and each $h < k$. 
			Recall that $\Ph(s) = \Ph_u(s) = \Id + \ph_u(s)$. By \Cref{banachfdb},
			\begin{align*}
			\frac{\p_x^k (u(s)\circ \Ph(s))(x)}{k!}
			&= 
			(\p_x u(s))(\Ph(s)(x))  \frac{\p_x^{k}\Ph(s)(x)}{k!} + R(s,x),		
			\end{align*}
			where $R(s,x)$ is the rest of the Fa\`a di Bruno formula which involves only derivatives $\p_x^j u(s)$ 
			of order $j\ge 2$ and derivatives $\p_x^h \Ph_u(t)$ of order $h \le k-1$. Using the induction hypothesis 
			it is easy to see that
			\[
				\int_0^1 \|R(s,x)\|_{L_k} \,ds =: C < \infty.
			\]
			With \eqref{sol} we have 
			\[
				\|\p_x^k\ph_u(t,x)\|_{L_k} \le \int_0^t \|(\p_x u)(s,x+ \ph_u(s,x))\|_{L_1}  \frac{\|\p_x^{k}(x + \ph_u(s,x))\|_{L_k}}{k!}  ds + C.
			\]
			Then Gronwall's inequality implies that $\p_x^k\vh_u(t)$ is globally bounded, uniformly in $t$. 
			That $t \mapsto \ph_u(t) \in \cB(\R^d,\R^d)$ is continuous is easily shown in a manner similar  
			to the proof of \Cref{claim:cont} in \Cref{MinfRoumieu} (the proof actually simplifies since each derivative 
			can be treated separately). 

			The remaining cases $\cA = W^{\infty,p}$, for $p<\infty$, $\cA = \cS$, and $\cA=\cD$ 
			can be proved with slight modifications of 
			the arguments used in the proofs of \Cref{MpRoumieu}, \Cref{gelfandshilovroumieuprop}, and 
			\Cref{Roumieucomp}, since $\cA$ is
			continuously included in $\cB$. 
		\end{proof}

\section{Continuity of the flow map} \label{sec:continuity}

In this section we prove that the map $u \mapsto \ph_u$ is continuous for all classical test function spaces.
We do not know if  similar results hold for the ultradifferentiable classes.  

For $f \in \Sp(\R^d,\R^d)$ we will write
\[
	\|f\|_{W^{k,p}} := \sum_{j=0}^k \Big(\int_{\R^d} \|f^{(j)}(x)\|_{L_j(\R^d,\R^d)}^p \, dx \Big)^{1/p}
\]
in the following.

\begin{lemma} \label{lem:Sobolev1}
	Let $p \in [1,\infty]$.
	Let $g,f \in W^{\infty,p}(\R^d,\R^d)$ and assume that $\Id + f$ is a diffeomorphism of $\R^d$
	with $c:= \inf_{x \in \R^d} \det d(x +f(x)) >0$. 
	Then, for all $k \in \N$ there is a constant $C=C(c,k)>0$, such that
	\[
		\|g \o (\Id + f) \|_{W^{k,p}} \le C \|g\|_{W^{k,p}} (1  + \|f\|_{W^{k,\infty}})^k.	
	\]
\end{lemma}

\begin{proof}
	We prove the assertion by induction on $k$.
	By assumption, $\|g \o (\Id + f)\|_{L^p} \le c^{-1/p} \|g\|_{L^p}$, thus the assertion holds for $k=0$.
	For $k=1$ we have,
	\begin{align*}
		\|d(g\circ(\Id+f))\|_{L^p}
		&\le   \|dg\circ(\Id+f)\|_{L^p} + \|dg\circ(\Id+f)\cdot df\|_{L^p}
		\\
		&\le c^{-1/p} \|dg\|_{L^p} + c^{-1} \|dg\|_{L^p} \|df\|_{L^\infty} 
		\\
		&\le c^{-1/p}\|g\|_{W^{1,p}} (1  + \|f\|_{W^{1,\infty}}).
	\end{align*}
			Now assume the statement holds for $k-1$. Then 
			\begin{align*}
		\|d(g\circ(\Id+f))\|_{W^{k-1,p}}
		&\le   \|dg\circ(\Id+f)\|_{W^{k-1,p}} + \|dg\circ(\Id+f)\cdot df\|_{W^{k-1,p}}
		\\
		&\le C\|dg\|_{W^{k-1,p}} (1  + \|f\|_{W^{k-1,\infty}})^{k-1} (1 + \|df\|_{W^{k-1,\infty}}) 
		\\
		&\le C \|g\|_{W^{k,p}} (1  + \|f\|_{W^{k,\infty}})^k. \qedhere
	\end{align*}
\end{proof}

\begin{lemma} \label{lem:Sobolev2}
	Let $p \in [1,\infty]$.
	Let $g,f_1,f_2 \in W^{\infty,p}(\R^d,\R^d)$ and assume that $\Id + f_i$ are diffeomorphisms of $\R^d$
	with $c_i:= \inf_{x \in \R^d} \det d(x +f_i(x)) >0$. 
	Then, for all $k \in \N$ there is a constant $C=C(c_i,k)>0$, such that
	\[
		\|g \o (\Id + f_1) - g \o (\Id + f_2) \|_{W^{k,p}} \le C \|g\|_{W^{k+1,\infty}} (1  +\max_{i=1,2} \|f_i\|_{W^{k,\infty}})^k \|f_1 - f_2\|_{W^{k,p}}. 	
	\]
\end{lemma}

\begin{proof}
	Induction on $k$. For $k=0$,
	\begin{align*}
		|g \o (\Id + f_1) - g \o (\Id + f_2) | \le \|dg\|_{L^\infty} |f_1 -f_2| 
	\end{align*}
	and hence $\|g \o (\Id + f_1) - g \o (\Id + f_2) \|_{L^p} \le \|g\|_{W^{1,\infty}} \|f_1 -f_2\|_{L^p}$. 
	Suppose that the claim holds for $k-1$. 
	Then, using the induction hypothesis and \Cref{lem:Sobolev1} for $p =\infty$, 
	\begin{align*}
		\MoveEqLeft
		\|d(g \o (\Id + f_1)) - d(g \o (\Id + f_2)) \|_{W^{k-1,p}} 
		\\
		&\le  \|(dg \o (\Id + f_1) - dg \o (\Id + f_2)) (\mathbb 1 + df_2)\|_{W^{k-1,p}} 
		\\&\quad + 
		\|dg \o (\Id + f_1)(df_1-df_2)\|_{W^{k-1,p}} 
		\\
		&\le 
		C \|dg\|_{W^{k,\infty}} (1  +\max_{i=1,2} \|f_i\|_{W^{k-1,\infty}})^{k-1} 
		\|f_1 - f_2\|_{W^{k-1,p}}(1 + \|f_2\|_{W^{k-1,\infty}}) 
		\\&\quad + 
		 \|dg \o (\Id + f_1)\|_{W^{k-1,\infty}}  \|df_1-df_2\|_{W^{k-1,p}} 
		\\
		&\le 
		C \|dg\|_{W^{k,\infty}} (1  +\max_{i=1,2} \|f_i\|_{W^{k-1,\infty}})^{k-1} 
		\|f_1 - f_2\|_{W^{k-1,p}}(1 + \|f_2\|_{W^{k-1,\infty}}) 
		\\&\quad + 
		C \|dg\|_{W^{k-1,\infty}} (1 + \|f_1\|_{W^{k-1,\infty}})^{k-1}  \|df_1-df_2\|_{W^{k-1,p}}
		\\
		&\le 
		2C \|g\|_{W^{k+1,\infty}} (1  +\max_{i=1,2} \|f_i\|_{W^{k,\infty}})^k \|f_1 - f_2\|_{W^{k,p}}. \qedhere
	\end{align*}
\end{proof}

\begin{theorem} \label{thm:contflowmap}
  Let $1 \le p \le \infty$.
  The mapping 
  \begin{equation} \label{Xtomu}
    L^1([0,1],\Sp(\R^d,\R^d)) \ni u \mapsto \ph_u \in C([0,1],\Sp(\R^d,\R^d))
  \end{equation}
  is continuous.
\end{theorem}

\begin{proof}
	The case $p=\infty$ follows immediately from \cite[Theorem 5.6]{NenningRainer16}
	(the assumption that the vector fields vanish together with all its derivatives as $|x| \to \infty$ 
	was not used in the proof).

	Let $p \in [1,\infty)$. 
	Then $L^1([0,1],\Sp(\R^d,\R^d))$ is continuously included in $L^1([0,1],\cB(\R^d,\R^d))$.
	The result for $p=\infty$
	together with \cite[Proposition 3.6]{NenningRainer16} shows that both $u \mapsto \Ph_u$ and
	$u \mapsto \Ph_u^{-1}$ are bounded into $C(I,\Diff_0 \cB)$. 
	And thus the arguments after ($\bullet \bullet$) on p.\ \pageref{detpositive} imply that for every bounded set $U \subseteq L^1(I,\cB(\R^d,\R^d))$ 
	there is a constant $c>0$ such that 
	\begin{equation} \label{eq:lowerbound}
		\det d\Ph_u(t,x) \ge c 
	\end{equation}
	for all $x\in \R^d$, $t \in I$, and $u \in U$.

  Let $u,v$ be in a bounded subset of $L^1([0,1],\Sp(\R^d,\R^d))$ 
  and let $\ph_u,\ph_v \in C([0,1],\Sp(\R^d,\R^d))$ be the corresponding 
  flows. Thanks to \eqref{eq:lowerbound} we may apply
  \Cref{lem:Sobolev1} and \Cref{lem:Sobolev2} to obtain
	\begin{align*}
			\MoveEqLeft
			\|\ph_u(t) - \ph_v(t)\|_{W^{k,p}} 
			\le \int_0^t \|u(s)\circ\Ph_u(s) - v(s)\circ\Ph_v(s)\|_{W^{k,p}} \,ds
			\\
			& \le \int_0^t \|u(s)\circ\Ph_u(s) - u(s)\circ\Ph_v(s)\|_{W^{k,p}} + \|(u(s) - v(s))\circ\Ph_v(s)\|_{W^{k,p}} \,ds \\
			& \le C \int_0^t \|u(s)\|_{W^{k+1,\infty}} \|\ph_u(s) - \ph_v(s)\|_{W^{k,p}} 	
			+ \|u(s) - v(s)\|_{W^{k,p}} \,ds	
	\end{align*}
	for a constant $C$ independent of $u,v$.
	Then Gronwall's inequality implies
	\[
	\|\ph_u(t) - \ph_v(t)\|_{W^{k,p}} \le C_1 \|u- v\|_{L^1(I,W^{k,p})} \exp (C_2\|u\|_{L^1(I,W^{k+1,\infty})}),
	\]
	and consequently, $\|\ph_u - \ph_v\|_{C(I,{W^{k,p}})} \le C_3 \|u- v\|_{L^1(I,W^{k,p})}$. 
	This implies the result. 
\end{proof}

\begin{theorem}
  The mappings 
  \begin{align*} 
    L^1([0,1],\cS(\R^d,\R^d)) \ni u &\mapsto \ph_u \in C([0,1],\cS(\R^d,\R^d)),
    \\
    L^1([0,1],\cD_K(\R^d,\R^d)) \ni u &\mapsto \ph_u \in C([0,1],\cD(\R^d,\R^d))
  \end{align*}
  are continuous.
\end{theorem}

\begin{proof}
	The case $\cD$ is immediate from the case $\cB$. 
	
	It is easy to see that for $g,f,f_1,f_2 \in \cS(\R^d,\R^d)$ we have the following analogues of 
	\Cref{lem:Sobolev1} and \Cref{lem:Sobolev2}:   	
	\begin{align*}
		\|g \o (\Id + f) \|^{(p,k)} &\le\! C \|g\|^{(p,k)} (1  + \|f\|_{W^{k,\infty}})^k,
		\\
		\|g \o (\Id + f_1) - g \o (\Id + f_2) \|^{(p,k)} &\le\! C \|g\|_{W^{k+1,\infty}} (1  +\max_{i=1,2} \|f_i\|_{W^{k,\infty}})^k \|f_1 - f_2\|^{(p,k)}.
	\end{align*}
	Then the case $\cS$ can be proved in a similar way as \Cref{thm:contflowmap}.
\end{proof}

	\section{Application: The Bergman space on the polystrip is ODE-closed}
	\label{sec:Bergman}

		In this section we prove that the $p$-Bergman space $A^p(S_{(r)})$ 
		on the polystrip $S_{(r)} := \{|\Im(z)| <r\}^d \subseteq \C^d$ 
		is continuously included in $W^{\mathbf{1}, p}_{a/r}(\R^d,\C)$ and continuously  
		contains $W^{\mathbf{1}, p}_{b/r}(\R^d,\C)$ if $a,b$ are suitable constants, 
		where $\mathbf{1} = (1,1,\ldots)$. 
		As an application we obtain that the scale of $p$-Bergman spaces on the polystrip 
		with variable width is ODE-closed.

		\subsection{Bergman spaces on polystrips}
	
		Let $S_r := \{z \in \C: |\Im(z)|<r\}$ be the horizontal open strip centered at the real line with width $2r>0$.
		For $r_1, \dots , r_d > 0$ consider the polystrip 
		$S_{(r_i)} := S_{r_1} \times \cdots \times S_{r_d}$, in particular, $S_r^d = S_{(r)}$. 
		For $p \in [1,\infty]$, we consider the \emph{$p$-Bergman space}
		\[
		A^p(S_{(r_i)}):= \cH(S_{(r_i)})\cap L^p(S_{(r_i)}),
		\]
		i.e., the space of holomorphic $L^p$-functions on $S_{(r_i)}$.
		Endowed with the $L^p$-norm, it is a complex Banach space, which is a direct consequence of the following Lemma.
		
		\begin{lemma}
			\label{Bergmancomplete}
			Let $0<l_i<r_i$ for $1\le i \le d$. Then there exists $C>0$ such that 
			\[
			\|F\|_{L^\infty(S_{(l_i)})} \le C \|F\|_{L^p(S_{(r_i)})} \quad \text{ for } F \in A^p(S_{(r_i)}).
			\]
		\end{lemma}
		
		\begin{proof}
			This is an easy consequence of the mean value inequality for subharmonic functions applied to the separately 
			subharmonic function $|F|^p$.
		\end{proof}
		In addition, we consider the \emph{real $p$-Bergman space}
		\[
		A^p_\R(S_{(r_i)}):= \big\{F \in A^p(S_{(r_i)}): ~ F(\R^d)\subseteq \R\big\}.
		\]
		Endowed with the $L^p$-norm, it is a real Banach space, which is again a direct consequence of \Cref{Bergmancomplete}.\par

		We are interested in the inductive limits of the Bergman spaces with respect to all positive widths of the 
		underlying polystrip
		\begin{gather*}
			\underrightarrow A^p(\R^d):= \varinjlim_{r > 0} A^p(S_{(r)}),\\
			\underrightarrow A^p_\R(\R^d):= \varinjlim_{r > 0} A^p_\R(S_{(r)}),
		\end{gather*}
		as well as in the projective limits
		\begin{gather*}
			\underleftarrow A^p(\R^d):= \varprojlim_{r > 0} A^p(S_{(r)}),\\
			\underleftarrow A^p_\R(\R^d):= \varprojlim_{r > 0} A^p_\R(S_{(r)}).
		\end{gather*}
		We set $A^p_\R(S_{(r_i)}, \R^n) := A^p_\R(S_{(r_i)})^n$,
		$\underrightarrow A^p_\R(\R^d, \R^n) := \underrightarrow A^p_\R(\R^d)^n$, 
		and $\underleftarrow A^p_\R(\R^d, \R^n) := \underleftarrow A^p_\R(\R^d)^n$.

		\subsection{The scale of Bergman spaces on polystrips with variable width is ODE-closed}
		
		\begin{theorem}
			\label{bergmaninclusion}
			Let $0<\si<1<\rh<\infty$, $p \in [1,\infty]$, and $r>0$. Then we have continuous inclusions 
			\[
      		\xymatrix{
      			W^{\mathbf{1}, p}_{\si/(2dr)}(\R^d,\C) \ar@{{ >}->}[r]^{\quad \cS} & A^p(S_{(r)}) 
      			\ar@{{ >}->}[r]^{\!\!\!\cR} 
      			& W^{\mathbf{1}, p}_{\rh/r}(\R^d,\C),
      		}
			\]
			where $\cR$ assigns to each function in the Bergman space the restriction to $\R^d$, whereas $\cS$ assigns to each function its unique holomorphic extension to the polystrip. 
			In particular, 
			\[
			W^{\{\mathbf{1}\},p}(\R^d,\C) \cong \underrightarrow A^p(\R^d) \quad \text{ and } \quad 
			W^{(\mathbf{1}),p}(\R^d,\C) \cong \underleftarrow A^p(\R^d).
			\]
		\end{theorem}
		
		\begin{corollary}
			\label{contemb}
			In the setting of \Cref{bergmaninclusion},
			\[
			\xymatrix{
      			W^{\mathbf{1}, p}_{\si/(2dr)}(\R^d,\R) \ar@{{ >}->}[r]^{\quad \cS} & A^p_\R(S_{(r)}) 
      			\ar@{{ >}->}[r]^{\!\!\!\cR} 
      			& W^{\mathbf{1}, p}_{\rh/r}(\R^d,\R),
      		}
			\]
			in particular, $W^{\{\mathbf{1}\},p}(\R^d,\R) \cong \underrightarrow A^p_\R(\R^d)$ and 
			$W^{(\mathbf{1}),p}(\R^d,\R) \cong \underleftarrow A^p_\R(\R^d)$.	
		\end{corollary}

		As an application, we obtain that the spaces $\underrightarrow A^p_\R(\R^d, \R^d)$ 
		and $\underleftarrow A^p_\R(\R^d, \R^d)$ are ODE-closed in the following sense:
		In analogy to \Cref{sec:diffgroups}, let $\cF_{\underrightarrow A^p_\R(\R^d,\R^d)}$ be the set of functions 
		$u: I \rightarrow \underrightarrow A^p_\R(\R^d,\R^d)$, that actually map into some step $A^p_\R(S_{(r)},\R^d)$ and are Bochner integrable therein, in particular,
		\[
		\exists r > 0: \int_0^1 \|u(t)\|_{L^p(S_{(r)})}\,dt < \infty.
		\]
		Moreover, let $\cF_{\underleftarrow A^p_\R(\R^d,\R^d)} := L^1(I,\underleftarrow A^p_\R(\R^d,\R^d))$ so that 
		\[
		\forall r > 0: \int_0^1 \|u(t)\|_{L^p(S_{(r)})}\,dt < \infty.
		\]

		\begin{theorem}
			We have:
			\begin{enumerate}
				\item For each $u \in \cF_{\underrightarrow A^p_\R(\R^d,\R^d)}$ the flow $t \mapsto \ph_u(t)$ is a 
				continuous curve in $\underrightarrow A^p_\R(\R^d,\R^d)$.
				\item For each $u \in \cF_{\underleftarrow A^p_\R(\R^d,\R^d)}$ the flow $t \mapsto \ph_u(t)$ is a 
				continuous curve in $\underleftarrow A^p_\R(\R^d,\R^d)$.
			\end{enumerate}
		\end{theorem}
		
		\begin{proof}
			(1) follows from \Cref{contemb}, \Cref{MinfRoumieu}, and \Cref{MpRoumieu}.
			(2) is a consequence of \Cref{contemb} and \Cref{BeurlingSobolev}. 
		\end{proof}

		To prove \Cref{bergmaninclusion}, we need some auxiliary results.

		\begin{lemma}
			\label{integralestimate}
			There exists $D > 0$ such that for all $\al \in \N_{\ge 1}^d$,
			\[
			\int_{\R^d} \frac{1}{\prod_{j = 1}^d |x_j + iy_j|^{\al_j+1}} \,d\la(x) \le \frac{D}{\prod_{j = 1}^d |y_j|^{\al_j}},
			\]
			where $\la$ denotes the Lebesgue measure in $\R^d$.
		\end{lemma}
		
		\begin{proof}
			Substitute $\xi_j = x_j/y_j$.
		\end{proof}

		\begin{proof}[Proof of \Cref{bergmaninclusion}]
			First we show
			continuity of $\cS$. 
			Fix $d \in \N$, $p \in [1,\infty]$, $r>0$, and $0<\si<1$. By \Cref{SobolevInequality}, 
			there exists a constant $C_1=C_1(p,d)$ such that for all $f \in C^\infty(\R^d)$, 
			$\be \in \N^d$, and $x \in \R^d$, 
			\[
			|\p^\be f (x)| \le C_1\sum_{|\al| \le l} \|\p^{\be+\al} f\|_{L^p},
			\]
			where $l:= \lfloor \frac{d}{p} \rfloor+1$. Therefore, for all $f \in W^{\mathbf{1}, p}_{\si/(2dr)}(\R^d, \C)$,
			\begin{align*}
			|\p^\be f(x)| &\le C_1\sum_{|\al|\le l} \|f\|^{\mathbf{1}, p}_{\si/(2dr)}
			\Big(\frac{\si}{2dr}\Big)^{|\be|+|\al|} (|\be|+|\al|)!
			\\
			&
						\le C_1  \|f\|^{\mathbf{1}, p}_{\si/(2dr)} \Big(\frac{\si}{dr}\Big)^{|\be|} |\be|! 
						\Big(\sum_{|\al|\le l} \Big(\frac{\si}{dr}\Big)^{|\al|} |\al|!\Big).
			\end{align*}
			Thus $f$ admits a holomorphic extension $F$ on $S_{(r)}$ with  
			\[
			F(z) = \sum_{\al \in \N^d} \p^\al f(x_1, \dots, x_d) \prod_{j =1}^d \frac{(iy_j)^{\al_j}}{\al_j}, \quad 
			z = (x_j +iy_j)_{j = 1}^d \in S_{(r)}. 
			\]
			Applying H\"older's inequality and \eqref{multinomial}, we thus get
			\begin{align*}
				|F(z)|^p \le \Big(\sum_{\al} \frac{|\p^\al f(x)|}{|\al|!} (dr)^{|\al|} \Big)^p \le 
				\Big(\sum_{\al} \Big(\frac{|\p^\al f(x)|}{|\al|!}\Big)^p \Big(\frac{dr}{\sqrt \si}\Big)^{p|\al|}\Big) 
				\Big(\sum_{\al} \sqrt \si^{q|\al|} \Big)^{p/q} 
			\end{align*}
			if $p \in (1,\infty)$ and $1/p + 1/q=1$.
			Thus,	
			\begin{align*}
			\int_{\R^d} |F(z)|^p \,d\la(x)
			\le C_2 \sum_{\al \in \N^d} \big(\|f\|^{\mathbf{1}, p}_{\si/(dr)}\big)^p \Big(\frac{\si}{dr}\Big)^{p|\al|}
			\Big(\frac{dr}{\sqrt \si}\Big)^{p|\al|} 
			&\le C_3
			\big(\|f\|^{\mathbf{1}, p}_{\si/(dr)}\big)^p.
			\end{align*}
			Then integration with respect to the $y_j$ over $(-r,r)$ implies continuity of $\cS$.
			The cases $p=1,\infty$ are seen similarly.

			Now let us consider $\cR$.
			Let $0<y_j<r$ for $1\le j \le d$ and consider the box $B(t,u):= [-t,t]+i[-u,u]$. 
			Let $F \in A^p(S_{(r)})$ and set $f:= F|_{\R^d}$. Then for $x \in \R^d$, $\al \in \N^d$, and $t>0$ 
			sufficiently large, we get (by iterating the Cauchy integral formula)
			\begin{align} \label{eq:Cauchy}
			&\p^{\al}f(x)
			=\frac{\al!}{(2\pi i)^d} \int_{\p B(t,y_1)} \dots \int_{\p B(t,y_d)} \frac{F(\ze_1, \dots, \ze_d)}{(\ze_1-x_1)^{\al_1+1} \cdots (\ze_d-x_d)^{\al_d+1}} \, d\ze_d \dots d\ze_1.
			\end{align}
			If we let $t$ tend to $\infty$ the contribution of the integrals over the vertical boundary parts 
			of the boxes $B(t,y_j)$ tends to $0$. 
			In order to estimate the contribution of the horizontal parts, observe that, by \Cref{integralestimate},
			for $1/p + 1/q=1$ and $k \ge 1$, 
			\begin{align*}
				\int_{-\infty}^\infty \frac{|G(t + i u)|}{|t - x + i u|^{k+1}} \,dt \le 
				\Big(\int_{-\infty}^\infty \frac{|G(t + i u)|^p}{|t - x + i u|^{k+1}} \,dt\Big)^{1/p}
				\Big(\frac{D}{|u|^k}\Big)^{1/q},	
			\end{align*} 
			and an analogous estimate holds if we integrate with respect to $x$ instead of $t$.
			Applying this to \eqref{eq:Cauchy} we find, for $\al \in \N^d_{\ge 1}$,  
			\begin{align*}
			\MoveEqLeft
			\|\p^\al f\|_{L^p(\R^d)}^p \prod_{j=1}^d |y_j|^{p\al_j}\\ &\le C_4 |\al|!^p \sum_{\ta \in \{0,1\}^d} \int_{-\infty}^{\infty}\dots \int_{-\infty}^{\infty}|F(((-1)^{\ta_j} \xi_j+i(-1)^{\ta_j+1}y_j)_{j=1}^d)|^p\,d\xi_d \dots d\xi_1.
			\end{align*}			
			Integration with respect to $y_j$ from $0$ to $r$ now yields
			\begin{align*}
			\|\p^\al f\|_{L^p(\R^d)}^p  \frac{r^{|\al|p+1}}{\prod_{j=1}^d (\al_jp+1)}\le C_4 |\al|!^p\|F\|_{L^p(S_{(r)})}^p,
			\end{align*}
			and thus, for $\rh >1$ so that $\rh^{-p |\al|} \prod_{j=1}^d (\al_jp+1)$ is bounded for all $\al$, 
			\[
			\sup_{\al \in \N_{\ge 1}^d} \frac{\|\p^\al f\|_{L^p(\R^d)}}{(\rh/r)^{|\al|} |\al|! } \le C_5  \|F\|_{L^p(S_{(r)})}.
			\]
			For $\al \in \N^d$ with some $\al_j = 0$, one uses the mean value inequalities for subharmonic functions for those zero-entries and the already established estimate for the non-zero entries to complete the proof.
		\end{proof}



\begin{thebibliography}{10}

\bibitem{Adams75}
R.~A. Adams, \emph{Sobolev spaces}, Academic Press, New York-London, 1975, Pure
  and Applied Mathematics, Vol. 65.

\bibitem{AulbachWanner96}
B.~Aulbach and T.~Wanner, \emph{Integral manifolds for {C}arath{\'e}odory type
  differential equations in {B}anach spaces}, Six lectures on dynamical systems
  ({A}ugsburg, 1994), World Sci. Publ., River Edge, NJ, 1996, pp.~45--119.
  
\bibitem{BM04}
E.~Bierstone and P.~D. Milman, \emph{Resolution of singularities in
  {D}enjoy-{C}arleman classes}, Selecta Math. (N.S.) \textbf{10} (2004), no.~1,
  1--28.

\bibitem{Blondia81}
C.~Blondia, \emph{Integration in locally convex spaces}, Simon Stevin
  \textbf{55} (1981), no.~3, 81--102. 

\bibitem{BruverisVialard14}
M.~Bruveris and F.-X. Vialard, \emph{On completeness of groups of
  diffeomorphisms}, J. Eur. Math. Soc. (JEMS) \textbf{19} (2017), no.~5,
  1507--1544.

\bibitem{Glockner16}
H.~Gl\"ockner, \emph{Measurable regularity properties of infinite-dimensional
  {L}ie groups},  (2016), ar{X}iv:1601.02568.

\bibitem{Komatsu73}
H.~Komatsu, \emph{Ultradistributions. {I}. {S}tructure theorems and a
  characterization}, J. Fac. Sci. Univ. Tokyo Sect. IA Math. \textbf{20}
  (1973), 25--105.

\bibitem{Komatsu79b}
\bysame, \emph{An analogue of the {C}auchy-{K}owalevsky theorem for
  ultradifferentiable functions and a division theorem for ultradistributions
  as its dual}, J. Fac. Sci. Univ. Tokyo Sect. IA Math. \textbf{26} (1979),
  no.~2, 239--254.

\bibitem{KM97}
A.~Kriegl and P.~W. Michor, \emph{The convenient setting of global analysis},
  Mathematical Surveys and Monographs, vol.~53, American Mathematical Society,
  Providence, RI, 1997, \url{http://www.ams.org/online\_bks/surv53/}.

\bibitem{KMRc}
A.~Kriegl, P.~W. Michor, and A.~Rainer, \emph{The convenient setting for
  non-quasianalytic {D}enjoy--{C}arleman differentiable mappings}, J. Funct.
  Anal. \textbf{256} (2009), 3510--3544.

\bibitem{KMRq}
\bysame, \emph{The convenient setting for quasianalytic {D}enjoy--{C}arleman
  differentiable mappings}, J. Funct. Anal. \textbf{261} (2011), 1799--1834.

\bibitem{KMRu}
\bysame, \emph{The convenient setting for {D}enjoy--{C}arleman differentiable
  mappings of {B}eurling and {R}oumieu type}, Rev. Mat. Complut. \textbf{28}
  (2015), no.~3, 549--597. 

\bibitem{KrieglMichorRainer14a}
\bysame, \emph{An exotic zoo of diffeomorphism groups on {$\mathbb{R}^n$}},
  Ann. Global Anal. Geom. \textbf{47} (2015), no.~2, 179--222. 

\bibitem{KrieglMichorRainer16}
\bysame, \emph{The exponential law for spaces of test functions and
  diffeomorphism groups}, Indag. Math. (N.S.) \textbf{27} (2016), no.~1,
  225--265. 

\bibitem{MichorMumford13}
P.~W. Michor and D.~Mumford, \emph{A zoo of diffeomorphism groups on
  {$\mathbb{R}^n$}}, Ann. Global Anal. Geom. \textbf{44} (2013), no.~4,
  529--540. 

\bibitem{NenningRainer16}
D.~N. Nenning and A.~Rainer, \emph{On groups of h\"older diffeomorphisms and
  their regularity}, to appear in Trans. Amer. Math. Soc. (2016).

\bibitem{RainerSchindl14}
A.~Rainer and G.~Schindl, \emph{Equivalence of stability properties for
  ultradifferentiable function classes}, Rev. R. Acad. Cienc. Exactas Fis. Nat.
  Ser. A Math. RACSAM. \textbf{110} (2016), no.~1, 17--32.

\bibitem{RainerSchindl16a}
\bysame, \emph{Extension of {W}hitney jets of controlled growth}, Math. Nachr.
  (2017), doi:10.1002/mana.201600321.

\bibitem{Schwartz66}
L.~Schwartz, \emph{Th{\'e}orie des distributions}, Publications de l'Institut
  de Math{\'e}matique de l'Universit{\'e} de Strasbourg, No. IX-X. Nouvelle
  {\'e}dition, enti{\'e}rement corrig{\'e}e, refondue et augment{\'e}e,
  Hermann, Paris, 1966. 

\bibitem{Trouve95}
A.~Trouv{\'e}, \emph{An infinite dimensional group approach for physics based
  models in pattern recognition}, {\tt
  http://cis.jhu.edu/publications/papers\_in\_database/alain/trouve1995.pdf},
  1995.

\bibitem{Yamanaka89}
T.~Yamanaka, \emph{Inverse map theorem in the ultra-{$F$}-differentiable
  class}, Proc. Japan Acad. Ser. A Math. Sci. \textbf{65} (1989), no.~7,
  199--202.

\bibitem{Younes10}
L.~Younes, \emph{Shapes and diffeomorphisms}, Applied Mathematical Sciences,
  vol. 171, Springer-Verlag, Berlin, 2010. 

\end{thebibliography}

\def\cprime{$'$}
\providecommand{\bysame}{\leavevmode\hbox to3em{\hrulefill}\thinspace}
\providecommand{\MR}{\relax\ifhmode\unskip\space\fi MR }
\providecommand{\MRhref}[2]{%
  \href{http://www.ams.org/mathscinet-getitem?mr=#1}{#2}
}
\providecommand{\href}[2]{#2}

\end{document}